\documentclass[11pt,reqno]{amsart}
\usepackage{hyperref,amsmath,amsfonts,mathscinet,amssymb,amsthm}
\usepackage[mathcal]{euscript}
\usepackage{enumitem}
\usepackage{xcolor}
\usepackage{graphics}
\usepackage{mathrsfs}
\usepackage[center]{caption}

\usepackage{todonotes}

\usepackage[normalem]{ulem}
\usepackage{soul}
\usepackage{cancel}
\newcommand\Del[1]{{\color{red}\ifmmode\cancel{#1}\else\sout{#1}\fi}}

\setlist[enumerate]{label={\rm(\roman*)}}

\theoremstyle{plain}
\newtheorem{thm}{Theorem}
\newtheorem{proposition}[thm]{Proposition}
\newtheorem{lemma}[thm]{Lemma}

\newtheorem{cor}[thm]{Corollary}

\theoremstyle{definition}
\newtheorem{defn}[thm]{Definition}

\newtheorem{rem}[thm]{Remark}
\newtheorem{example}[thm]{Example}

\numberwithin{thm}{section}
\numberwithin{equation}{section}
\expandafter\let\expandafter\oldproof\csname\string\proof\endcsname
\let\oldendproof\endproof
\renewenvironment{proof}[1][\proofname]{%
  \oldproof[\bf #1]%
}{\oldendproof}

\def\M{\mathscr M}

\def\N{\mathbb N}

\def\R{\mathbb R}

\newcommand{\RR}{\mathscr{R}}
\newcommand{\abs}[1]{\left|{#1}\right|}

\newtoks\by
\newtoks\paper
\newtoks\book
\newtoks\jour
\newtoks\yr
\newtoks\pages
\newtoks\vol
\newtoks\publ
\newtoks\eds
\newtoks\proc
\newtoks\no
\def\ota{{\hbox{???}}}
\def\cLear{\by=\ota\paper=\ota\book=\ota\jour=\ota\yr=\ota
\pages=\ota\vol=\ota\publ=\ota}
\def\endpaper{\the\by, \textit{\the\paper},
{\the\jour} \textbf{\the\vol} (\the\yr), \the\pages.\cLear}
\def\endbook{\the\by, \textit{\the\book}, \the\publ.\cLear}
\def\endprep{\the\by, \textit{\the\paper}, \the\jour.\cLear}
\def\endproc{\the\by, \textit{\the\paper}, \the\publ, \the\pages.\cLear}
\def\name#1#2{#1 #2}
\def\et{ and }

\begin{document}

\title[Characterization of functions with zero traces \dots]{Characterization of functions with zero traces via the distance function and Lorentz spaces}

\author{Ale\v s Nekvinda and Hana Tur\v cinov\'a}

\address{
 Ale\v s Nekvinda, Czech Technical University in Prague, Faculty of Civil Engineering, Department of Mathematics, Th\'akurova~7, 166~00 Praha~6, Czech Republic}
\email{ales.nekvinda@cvut.cz}
\urladdr{0000-0001-6303-5882}

\address{Hana Tur\v{c}inov\'{a}, Charles University, Faculty of Mathematics and Physics, Department of Mathematical Analysis, Sokolovsk\'a~83, 186~75 Praha~8, Czech Republic}
\email{turcinova@karlin.mff.cuni.cz}
\urladdr{0000-0002-5424-9413}


\thanks{This research was supported in part by the grant P201-18-00580S of the Czech Science Foundation.
The research of Ale\v s Nekvinda was supported by the European Regional Development Fund, project No.~CZ 02.1.01/0.0/0.0/16-019/0000778.
The research of H.~Tur\v cinov\'{a} was supported in part by the Charles University, project GA UK No. 358120, by the Primus research programme PRIMUS/21/SCI/002 of Charles University, by the Danube Region Grant No.~8X2043 of the Czech Ministry of Education, Youth and Sports, and Charles University Research program No.~UNCE/SCI/023.}

\begin{abstract}
Consider a regular domain $\Omega \subset \R^N$ and let $d(x)=\operatorname{dist}(x,\partial\Omega)$. Denote $L^{1,\infty}_a(\Omega)$ the space of functions from $L^{1,\infty}(\Omega)$ having absolutely continuous quasinorms. This set is essentially smaller than $L^{1,\infty}(\Omega)$ but, at the same time, essentially larger than a union of all $L^{1,q}(\Omega)$, $q\in[1,\infty)$.

A classical result of late 1980's states that for $p\in (1,\infty)$ and $m \in \mathbb{N}$, $u$ belongs to the Sobolev space  $W^{m,p}_0(\Omega)$ if and only if $u/d^m\in L^p(\Omega)$ and $\left|\nabla^m u\right|\in L^p(\Omega)$. During the consequent decades, several authors have spent considerable effort in order to relax the characterizing condition. Recently, it was proved that $u\in W^{m,p}_0(\Omega)$ if and only if $u/d^m\in L^1(\Omega)$ and $\left|\nabla^m u\right|\in L^p(\Omega)$. In this paper we show that for $N\geq1$ and $p\in(1,\infty)$ we have $u\in W^{1,p}_0(\Omega)$ if and only if $u/d\in L^{1,\infty}_a(\Omega)$ and $\left|\nabla u\right|\in L^p(\Omega)$. Moreover, we present a counterexample which demonstrates that after relaxing the condition $u/d\in L^{1,\infty}_a(\Omega)$ to $u/d\in L^{1,\infty}(\Omega)$ the equivalence no longer holds.
\end{abstract}

\date{\today}

\maketitle

\bibliographystyle{alpha}

\section{Introduction}

Sobolev spaces play an outstanding role in modern analysis and enjoy a wide array of applications mainly in partial differential equations, calculus of variations and mathematical physics. For $m,N\in\N$, $p\in[1,\infty)$ and an open subset $\Omega$ of the ambient Euclidean space $\R^N$, the classical Sobolev space $W^{m,p}(\Omega)$ is given as the collection of all $m$-times weakly differentiable functions $u\colon \Omega\to\R$ such that $|D^{\alpha}u|\in L^p(\Omega)$ for every multiindex $\alpha$ of order not exceeding $m$, in which $|v|$ stands for the Euclidean norm of the vector $v\in\R^N$. In connection with certain specific tasks, most notable of which is the Dirichlet problem, the subset $W^{m,p}_0(\Omega)$ of $W^{m,p}(\Omega)$, sheltering those functions $u\in W^{m,p}(\Omega)$, which in an appropriate sense vanish on the boundary of $\Omega$, is of crucial importance.

The literature is not unique as for the precise definition of $W^{m,p}_0(\Omega)$, but we can safely say that $W^{m,p}_0(\Omega)$ is classically defined as a closure in $W^{m,p}(\Omega)$ of the set of smooth functions having compact supports in $\Omega$. This definition is somewhat theoretical and does not always necessarily meet the demands of numerous applications. Therefore, a considerable effort has been spent by many authors in order to describe such space in a more manageable way, which would possibly turn out to be more practical. One of the most effective approaches to this task makes use of the so-called distance function $d$, defined on $\Omega$ as the distance from its boundary. It is proved in \cite[Theorem V.3.4]{edmunds1987} that for any domain $\Omega\subset\R^N$ different from $\R^N$ itself, any order $m\in\N$ and every $p\in(1,\infty)$, the following implication holds:
\begin{equation*}
    \text{if $u\in W^{m,p}(\Omega)$ and $\frac{u}{d^m}\in L^p(\Omega)$, then $u\in W^{m,p}_0(\Omega)$.}
\end{equation*}
It is worth noticing that, as stated in~\cite[Remark V.3.5]{edmunds1987}, the proof of this result, due to D.J. Harris, is based on the boundedness of the Hardy--Littlewood maximal operator on Lebesgue spaces, and therefore it requires $p>1$. The authors however mention that a result for $p=1$ is available, too, using the Whitney decomposition theorem instead. This approach is due to C. Kenig.

What is even more interesting is the fact that the implication can be reversed in a sense, that is, $u\in W^{m,p}_0(\Omega)$ implies $\frac{u}{d^m}\in L^p(\Omega)$. In the case when $\Omega$ has Lipschitz boundary, this result appears already in~\cite[Theorem~1]{kadlec1966}. For $p=2$, it is pointed out in~\cite[Remark X.6.8]{edmunds1987} for sufficiently regular domains (examples of which are given there in terms of certain modified cone conditions). The proof is based on the fact that, for regular domains, one can equivalently replace the distance function by the mean distance function. Further results in this direction can be found in~\cite{Bal:15},
in which the case of $p\in(1,\infty)$ is treated, and the equivalence of the distance function to the mean one is established for convex domains.

In summary, we have the equivalence
\begin{equation}\label{E:the-ekvivalence}
    u\in W^{1,p}_0(\Omega)\qquad\text{if and only if}\qquad \frac{u}{d}\in L^{p}(\Omega)\,\,\text{and}\,\,u\in W^{1,p}(\Omega),
\end{equation}
where $\Omega$ is a Lipschitz domain.

The characterization of the space $u\in W^{m,p}_0(\Omega)$ in terms of the distance function was later improved several times. An enormous effort has been spent on weakening the integrability conditions on $\frac{u}{d}$, especially for the first-order Sobolev spaces. In \cite{kinnunen1997}, which is mainly focused on relations between the Hardy inequality and the membership to $W^{1,p}_0(\Omega)$ of functions in $W^{1,p}(\Omega)$, it was shown that the $p$-integrability of $\frac{u}{d}$ can be in fact replaced by its corresponding weak integrability, and one still gets the same conclusion. Namely, for $p\in(1,\infty)$, one has,
\begin{equation}\label{E:the-ekvivalence-b}
    u\in W^{1,p}_0(\Omega)\qquad\text{if and only if}\qquad \frac{u}{d}\in L^{p,\infty}(\Omega)\,\,\text{and}\,\,u\in W^{1,p}(\Omega),
\end{equation}
in which the `if' part needs no restriction on the open set $\Omega$, while the `only if' part requires that $\R^N\setminus \Omega$ is uniformly $p$-fat (for the precise definition, see~\cite[p.~490]{kinnunen1997}). Note that this requirement is satisfied for instance by every Lipschitz domain, and also by every domain having the outer cone property.

The fact that the equivalence in~\eqref{E:the-ekvivalence} is preserved when the condition on $\frac{u}{d}$ is substantially weakened as in~\eqref{E:the-ekvivalence-b} may seem somewhat surprising. It demonstrates the strength of the condition $\left|\nabla u\right|\in L^p(\Omega)$. Note that this result also yields that $\frac{u}{d}\in L^{p,\infty}(\Omega)$ and $\left|\nabla u\right|\in L^p(\Omega)$ implies $\frac{u}{d}\in L^{p}(\Omega)$. We would like to point out that for regular domains one has $u\in W^{1,p}(\Omega)$ if and only if $|\nabla u|\in L^p(\Omega)$.

This leads us to an interesting general problem: Find as large structure $X(\Omega)$ as possible such that, for $p\in(1,\infty)$, the equivalence
$$
u\in W^{m,p}_0(\Omega)\qquad\text{if and only if}\qquad \frac{u}{d^m}\in X(\Omega)\,\,\text{and}\,\,\left|\nabla u\right|\in L^p(\Omega)
$$
holds.

In \cite{edmunds2015}, the choice $X(\Omega)=L^{p,\infty}(\Omega)$ was further weakened. Before stating the result, we first need to introduce a geometric property of a domain. We say that a bounded domain $\Omega$ has the \textit{outer ball portion property}, if there exist positive constants $b$ and $r_0$ such that, for all $x\in\partial\Omega$ and all $r \in \left(0, r_0\right]$, one has
\begin{equation*}
    \frac{\lambda^N\left(B(x,r)\cap (\R^N\setminus\Omega)\right)}{\lambda^N\left(B(x,r)\right)}\geq b.
\end{equation*}
We recall that domains enjoying this property can be also found in \cite{hajlasz1999}, but the actual term \textit{outer ball portion property} is for the first time being introduced here. Now we can return to our problem. It was proved in \cite{edmunds2015}, at least in the case $m=1$ and for $\Omega$ having the outer ball portion property, that one can choose $X(\Omega)=L^1(\Omega)$, i.e.
$$
u\in W^{1,p}_0(\Omega)\qquad\text{if and only if}\qquad \frac{u}{d}\in L^{1}(\Omega)\,\,\text{and}\,\,\left|\nabla u\right|\in L^p(\Omega).
$$
Moreover, more general treatise involving variable exponent spaces $L^{p(\cdot)}$ can be found in that paper. The last mentioned result was extended in \cite{edmunds2017} to the Sobolev spaces of higher order. More precisely, for $\Omega$ having the outer ball portion property, one has
$$
u\in W^{m,p(\cdot)}_0(\Omega)\qquad\text{if and only if}\qquad \frac{u}{d^m}\in L^1(\Omega)\,\,\text{and}\,\,\left|\nabla^m u\right|\in L^{p(\cdot)}(\Omega).
$$
Further development, this time using Lorentz spaces, can be traced to~\cite{MGRTurcinova}, in which it was proved that for $m=1$ and $q\in[1,\infty)$ we can take $X(\Omega)=L^{1,q}(\Omega)$, i.e.
$$
u\in W^{1,p}_0(\Omega)\qquad\text{if and only if}\qquad \frac{u}{d}\in L^{1,q}(\Omega)\,\,\text{and}\,\,\left|\nabla u\right|\in L^p(\Omega),
$$
where $\Omega$ is a domain with Lipschitz boundary.

The main goal of this paper is to prove, for $m=1$, that it suffices to choose $X(\Omega)=L^{1,\infty}_a(\Omega)$, where $L^{1,\infty}_a(\Omega)$ is the space of functions from $L^{1,\infty}(\Omega)$, whose norm is absolutely
continuous. We will also show that $X(\Omega)=L^{1,\infty}(\Omega)$ does not work, whence we provide a comprehensive answer in the scale of two-parametric Lorentz spaces, establishing moreover certain estimate of the position of
the threshold marking how far one can go with weakening of the condition on $\frac{u}{d}$. We would like to stress that the space $L^{1,\infty}_a(\Omega)$ is essentially larger than any $L^{1,q}(\Omega)$ for $q<\infty$, and, at the
same time, it is essentially smaller than $L^{1,\infty}(\Omega)$, so the result makes perfect sense.  Aside from these considerations, we shall also focus on the role of the regularity of $\Omega$. In doing so, we shall make
effective use of the so-called isoperimetric function of $\Omega$ (detailed definitions will be given below). Precisely, we will prove the following result.

\begin{thm}\label{T:main}
Let $p\in (1,\infty)$ and let $\Omega\subset\R^N$ be a domain obeying simultaneously the outer ball portion property and the estimate
\begin{equation}\label{E:iso}
  I_{\Omega}(s)\geq C s\quad\text{near $0$,}
\end{equation}
where $I_{\Omega}$ is the isoperimetric function of $\Omega$. Then
$$
\left|\nabla u\right| \in L^p(\Omega)\qquad\text{and}\qquad\frac{u}{d}\in L^{1,\infty}_a(\Omega)
$$
if and only if
$$
u\in W^{1,p}_0(\Omega).
$$
\end{thm}

To be more exact, for the implication
$$
  \frac{u}{d}\in L^{1,\infty}_a(\Omega)\quad\text{and}\quad \left|\nabla u\right|\in L^p(\Omega)\quad \Rightarrow\quad u\in W^{1,p}_0(\Omega)
$$
we need that $\Omega$ is~bounded and satisfies~\eqref{E:iso} and
\begin{equation}\label{E:intersection}
  W^{1,p}(\Omega)\cap W^{1,1}_0(\Omega)=W^{1,p}_0(\Omega),
\end{equation}
while for the converse one we repeat some of the assumptions on the regularity of a domain suitable for the validity of the appropriate Hardy inequality, which can be found in earlier literature, see e.g. \cite{kadlec1966}, \cite{edmunds1987}, \cite{kinnunen1997}, \cite{edmunds2015}. Altogether, it turns out that the requirement that $\Omega$ obeys the outer ball portion property is a reasonable universal condition which, on one hand, implies~\eqref{E:intersection}, while, on the other, had been used  as an assumption in~\cite{edmunds2015}.

We point out that some condition on regularity of domain will be needed for both implications. That constitutes a crucial difference from previous results in \cite{edmunds1987} and \cite{kinnunen1997}, where one implication was proved with no restriction on the regularity of a domain. We will give an example of a domain for which this implication fails.

Let us now describe the structure of the paper. In Section~\ref{S:preliminaries}, we collect all the necessary background material. We fix definitions here, and also most of the notation.

Section~\ref{S:proof} contains two key theorems that will be later used in the proof of the main result. First we prove a general theorem, and then we apply it to obtaining a certain implication between properties of functions.

In Section~\ref{S:domains}, relations between various requirements on the regularity of a domain are treated. Most of the results contained there are demonstrated with particular examples. The results contained in Sections~\ref{S:proof} and~\ref{S:domains} are harvested in Corollary~\ref{TH:equivalence} which states an assertion in spirit of Theorem~\ref{T:main}, allowing more general assumptions. The proof of Theorem~\ref{T:main} is located at the end of Section~\ref{S:domains}.

In Section~\ref{S:dimension1}, we focus in detail on the specific situation in dimension $N=1$. In Section~\ref{S:examples} we give two examples illustrating that the assumptions of the main theorems cannot be omitted.

We finish the paper with an Appendix, in which we collect various auxiliary results, some of them of independent interest. In particular,  Subsection~\ref{S:ACnorm} is devoted to background results about spaces of functions with absolute continuous norm, namely their equivalent definitions and their relations to Lorentz spaces. These results are applied at several parts of the paper, in particular in the proof of the key Theorem~\ref{TH:Lpinfty}. In Subsection~\ref{SS:isoperimetric}, we collect some useful knowledge concerning the isoperimetric function of specific domains. The results of this subsection are used for example in Section~\ref{S:domains}, in particular in the proof of Theorem~\ref{T:main}.

\section{Preliminaries} \label{S:preliminaries}

In this section, we collect definitions of objects of our study, fix notation and give a survey of concepts and results from theory of function spaces that will be used in the subsequent parts of the paper. Our standard general references are \cite{BS}, \cite{kufner1977}, \cite{mazja1985} and \cite{PKJF}, where more details can be found.

Let $(\mathscr{R},\mu)$ be a~non-atomic $\sigma$-finite measure space. We denote by $\mathscr{M}(\mathscr{R},\mu)$ the set of all $\mu$-measurable functions on $\mathscr{R}$ whose values lie
in $[-\infty,\infty]$, by $\mathscr{M}_+(\mathscr{R},\mu)$ the set of all functions in $\mathscr{M}(\RR,\mu)$ whose values lie in $[0,\infty]$, and by $\mathscr{M}_{0}(\mathscr{R},\mu)$ the set of all functions
in $\mathscr{M}(\mathscr{R},\mu)$ that are finite $\mu$-a.e.~on $\mathscr{R}$.

We denote by $\lambda^N$, $N\in\N$, the $N$-dimensional Lebesgue measure, $d\lambda^N(x)=dx$ and we denote by $\chi_E$ the characteristic function of a set $E$. For a function $f\in\M(\RR,\mu)$ we define two functions $f_+,f_-\in \M(\RR,\mu)$ such that $f_+=f\chi_{\{x\in\RR,f(x)>0\}}$ and $f_-=-f\chi_{\{x\in\RR,f(x)<0\}}$. If $A,B$ are two non-negative quantities, we shall write $A\approx B$ if there exist positive constants $c_1$ and $c_2$ independent of adequate parameters involved in $A$ and $B$ such that $c_1A\leq B\leq c_2A$. The convention $0\cdot \infty=0$ applies.

For $E\subset\R^N$, we denote by $\overline{E}$ the closure of $E$ and by $E^{\circ}$ the interior of $E$. Moreover, we write $\partial E = \overline{E}\setminus E^{\circ}$. We also define the \textit{distance function from the boundary} of $E$ as $d(x)=\operatorname{dist}(x,\partial E)$.

For $f\in\mathscr{M}(\mathscr{R},\mu)$, the function $f^*\colon [0,\infty)\to[0,\infty]$, defined by
\begin{equation*}
    f^{\ast}(t)=\inf\{\xi\geq 0:\mu(\{x\in \RR:\abs{f(x)}>\xi\})\leq t\}\quad\text{for $t\in[0,\infty)$,}
\end{equation*}
is called the \emph{non-increasing rearrangement} of $f$.

\begin{defn}[continuous embedding]
Let $X, Y$ be two quasinormed linear spaces and let $X\subset Y$.
We say that the space X is \emph{continuously embedded} into the space Y, denoted $X\hookrightarrow Y$, if there exists a constant $C\in(0,\infty)$ such that
$$
\left\|f\right\|_{Y}\leq C \left\|f\right\|_{X}\quad \text{for every $f\in X$}.
$$
The smallest possible constant $C$ is called a \emph{norm of the embedding}.
\end{defn}

\begin{defn}[Lebesgue spaces]
Let $p\in[1,\infty]$. The collection $L^p(\mathscr{R})=L^p(\mathscr{R},\mu)$ of all functions $f\in\mathscr{M}(\mathscr{R},\mu)$ such that
$\left\|f\right\|_{L^{p}(\mathscr{R})}<\infty$, where
\begin{align*}
\left\|f\right\|_{L^{p}(\mathscr{R})}=
\left\{
\begin{array}{l@{\quad}l}
\left(\int_{\mathscr{R}}|f|^p d\mu\right)^{\frac1p},& p\in[1,\infty),\\
\operatorname{ess~sup}_{\mathscr{R}}|f|,& p=\infty,
\end{array}
\right.
\end{align*}
is called the \textit{Lebesgue space}.
\end{defn}

\noindent Recall that if $(\RR,\mu)$ is a finite measure space, then the Lebesgue spaces are nested in the sense that $L\sp{p_1}(\RR)\hookrightarrow L\sp {p_2}(\RR)$ whenever $1\leq p_2\leq p_1\leq \infty$ with a norm of the embedding equal to $\mu(\RR)^{1/p_2-1/p_1}$.

Lebesgue spaces are a pivotal example of the so-called rearrangement-invariant Banach function spaces.

\begin{defn}
We say that a functional $\varrho\colon \mathscr{M}_{+}(\RR,\mu)\rightarrow [0,\infty]$ is a \emph{Banach function norm} if, for all $f$, $g$ and $\{f_n\}^{\infty}_{n=1}$ in $\mathscr{M}_{+}(\RR,\mu)$, for every $\lambda\in[0,\infty)$ and for every $\mu$-measurable subset $E$ of $\RR$, the following five properties are satisfied:

(P1) $\varrho(f)=0 \Leftrightarrow f=0$ $\mu$-a.e.~on $\mathscr{R}$; $\varrho(\lambda f)=\lambda\varrho(f)$; $\varrho(f+g)\leq
\varrho(f)+\varrho(g)$;

(P2) $g\leq f$ $\mu$-a.e.~on $\mathscr{R}$ $\Rightarrow\varrho(g)\leq\varrho(f)$;

(P3) $f_n\nearrow f$ $\mu$-a.e. on $\RR$ $\Rightarrow\varrho(f_n)\nearrow\varrho(f)$;

(P4) $\mu(E)<\infty \Rightarrow \varrho(\chi_{E})<\infty$;

(P5) $\mu(E)<\infty \Rightarrow \int_{E}f \,d\mu\leq C_E \varrho(f)$ for some constant $C_E \in (0,\infty)$ possibly depending on $E$ and $\varrho$ but independent of $f$.

\medskip

We say that $\varrho\colon \mathscr{M}_{+}(\RR,\mu)\rightarrow [0,\infty]$ is a \textit{Banach function quasinorm} if it satisfies (P2), (P3), and (P4), and (P1) replaced by its weakened modification (Q1), where

(Q1) $\varrho(f)=0\Leftrightarrow f=0$ $\mu$-a.e.~on $\mathscr{R}$, $\varrho(\lambda f)=\lambda\varrho(f)$ and there exists $C\in(0,\infty)$ such that
\begin{equation*}
    \varrho(f+g)\le C(\varrho(f)+\varrho(g))\quad\text{for every $f,g\in\mathscr{M}_{+}(\RR,\mu)$.}
\end{equation*}

For a Banach function~quasinorm $\varrho$ and $X=\{f\in \mathscr{M}_{0}(\RR,\mu), \varrho(|f|)<\infty\}$ we denote $\left\|f\right\|_{X}=\varrho(|f|)$ for $f\in\M(\RR,\mu)$ and we then say that $X$ is a \emph{quasi-Banach function space} over $(\RR,\mu)$. In the case $\varrho$
is a Banach function~norm, we call $X$ a \emph{Banach function space} over $(\RR,\mu)$.
\end{defn}

If $p\in[1,\infty]$, we define the \textit{H\"older conjugate exponent}, $p'$, of $p$, by
\begin{equation*}
    p'=
        \begin{cases}
            \infty &\text{if $p=1$,}
                \\
            \frac{p}{p-1} &\text{if $p\in(1,\infty)$,}
                \\
            1  &\text{if $p=\infty$.}
        \end{cases}
\end{equation*}

If $p\in[1,\infty]$, then the \textit{H\"older inequality} states that
\begin{equation}\label{E:holder}
\left\|fg\right\|_{L^1(\mathscr{R})}
\leq
\left\|f\right\|_{L^p(\mathscr{R})}\left\|g\right\|_{L^{p'}(\mathscr{R})}
\end{equation}
holds for every $f,g\in\M_+(\RR,\mu)$ such that $f\in L^p(\mathscr{R})$ and $g\in L^{p'}(\mathscr{R})$. Then also $fg\in L^1(\mathscr{R})$.


Now we will focus on a certain more general concept of function spaces than Lebesgue ones, namely on the so-called Lorentz spaces, which will have a~crucial importance in the formulation of our results.

\begin{defn}[Lorentz spaces]
Let $p,q\in[1,\infty]$. The collection $L^{p,q}(\mathscr{R})=L^{p,q}(\mathscr{R},\mu)$ of all functions $f\in \mathscr{M}_0(\mathscr{R},\mu)$ such that $\left\|f\right\|_{L^{p,q}(\mathscr{R})}<\infty$, where
$$
\left\|f\right\|_{L^{p,q}(\mathscr{R})}
=
\left\{
\begin{array}{l@{\quad}l}
\left(\int_0^{\infty}[t^{\frac{1}{p}}f^{*}(t)]^q\frac{dt}{t}\right)^{\frac{1}{q}},& q\in[1,\infty),\\
\sup_{0<t<\infty}t^{\frac1p}f^{*}(t), & q=\infty,
\end{array}
\right.
$$
is called a \emph{Lorentz space}.
\end{defn}

\noindent It will be useful to recall that $L\sp{p,p}(\RR)=L\sp p(\RR)$ for every $p\in[1,\infty]$, that $L\sp{p,q}(\RR)\hookrightarrow L\sp {p,r}(\RR)$ whenever $p\in[1,\infty]$ and $1\leq q\leq r\leq \infty$ with the norm of the embedding equal to $(p/q)^{1/q-1/r}$, and that, if moreover $\mu(\RR)<\infty$, then also $L\sp{p_1,q}(\RR)\hookrightarrow L\sp {p_2,r}(\RR)$ whenever $1\leq p_2<p_1 \leq \infty$  and $q,r\in[1,\infty]$. If either $p\in[1,\infty)$ and $q\in[1,\infty]$ or $p=q=\infty$, then $L^{p,q}(\RR)$ is a quasi-Banach function~space. If one of the  conditions
\begin{equation*}
\begin{cases}
p\in(1,\infty),\ q\in[1,\infty],\\
p=q=1,\\
p=q=\infty,
\end{cases}
\end{equation*}
holds, then $L\sp{p,q}(\RR)$ is equivalent to a Banach function~space.

\begin{rem}
A quasi-Banach function space may, or may not, satisfy (P5). A typical example of a quasi-Banach function~space which does not satisfy (P5) is $L^{1,q}(\RR)$ with $q\in(1,\infty]$. Moreover spaces $L^{1,q}(\RR)$ with $q\in(1,\infty]$ do not satisfy (P1) and there does not exist any equivalent norm satisfying (P1).
\end{rem}

\begin{rem}(Lorentz norm via distribution)\label{TH:Lorentz_eq}
The functional $\left\|\cdot\right\|_{L^{p,q}(\mathscr{R})}$ can be equivalently rewritten as
$$
\left\|f\right\|_{L^{p,q}(\mathscr{R})}=p^{\frac1q}\left\|\xi^{1-\frac1q}\mu(\left\{x\in\mathscr{R}:\left|f(x)\right|>\xi\right\})^{\frac{1}{p}}\right\|_{L^q(0,\infty)}.
$$
\end{rem}

\begin{defn}
Let $p\in[1,\infty)$. The set $L^{p,\infty}_a(\mathscr{R})=L^{p,\infty}_a(\mathscr{R},\mu)$ is defined as the collection of all functions $f\in L^{p,\infty}(\mathscr{R})$ having \textit{absolutely continuous norm}, i.e.
$$
\left\|f\chi_{E_k}\right\|_{L^{p,\infty}(\mathscr{R})}\rightarrow0\quad\text{for every sequence $\left\{E_k\right\}_{k=1}^{\infty}$ satisfying $E_k\rightarrow\emptyset$,}
$$
where $E_k\rightarrow\emptyset$ denotes the fact that $\chi_{E_k}\rightarrow 0$ $\mu$-a.e. on $\RR$.
\end{defn}

We will present more details about spaces of functions with absolute continuous norm in the Section~\ref{S:ACnorm}.

Let us recall the definition of the maximal operator.
\begin{defn}
Suppose that $R\in(0,\infty)$ and $\Omega\subset \R^N$. The \textit{maximal operator} $M_R$ is defined by
$$
M_Ru(x)=\sup_{r\in(0,R)}\frac1{\lambda^N(B(x,r))}\int_{B(x,r)\cap \Omega}\left|u(y)\right|\,dy
$$
for a functions $u$ integrable on each measurable bounded subset of $\Omega$. When $R=\infty$, the corresponding operator is denoted $M$ and it is called the \textit{Hardy-Littlewood maximal operator}.
\end{defn}
Note that for function $u$ integrable on each measurable bounded subset of $\Omega$ and $x \in \Omega$ we have that $M_Ru(x)\leq Mu(x)$ and the operator $M$ is bounded from $L^p(\Omega)$ to $L^p(\Omega)$, $p\in (1, \infty]$.

Another class of spaces of crucial importance in this research is that of the Sobolev spaces. Let us focus on its definition and properties.

\begin{defn}[Sobolev spaces]
Let $\Omega\subset\R^N$ be an open set, $m$ be a nonnegative integer and $1\leq p\leq\infty$. Set
$$
W^{m,p}(\Omega)=\{u\in L^p(\Omega)\colon D^{\alpha}u \in L^p(\Omega) \text{ for }0\leq|\alpha|\leq m\},
$$
where we denote by $\alpha=(\alpha_1,\dots,\alpha_N)$ a multiindex and by $D^{\alpha}u$ a weak derivative of~$u$ with respect to $\alpha$.
The set $W^{m,p}(\Omega)$ is called the \textit{Sobolev space}.
We define the functional $\left\|\cdot\right\|_{W^{m,p}(\Omega)}$ as follows:
\begin{align*}
&\left\|u\right\|_{W^{m,p}(\Omega)}=
\left\{
\begin{array}{l@{\quad}l}
\left(\sum_{0\leq|\alpha|\leq m}\left\|D^{\alpha}u\right\|^p_{L^p(\Omega)}\right)^{\frac1p},& 1\leq p<\infty,\\
\max_{0\leq|\alpha|\leq m}\left\|D^{\alpha}u\right\|_{L^{\infty}(\Omega)},& p=\infty,
\end{array}
\right.
\end{align*}
for every function $u$ for which the right-hand side is defined.

We define the set $W^{m,p}_0(\Omega)$ as the closure of $C^{\infty}_0(\Omega)$ in the space $W^{m,p}(\Omega)$, where $C^{\infty}_0(\Omega)$ denotes a set of all functions defined on $\Omega$ with continuous derivatives of each order and whose support is a compact subset of $\Omega$.
\end{defn}

The sets $W^{m,p}(\Omega)$ and $W^{m,p}_0(\Omega)$ equipped with the functional $\left\|\cdot\right\|_{W^{m,p}(\Omega)}$ are Banach spaces.
\begin{defn}[domain and several types of regularity]\label{TH:doamins}
~
\begin{itemize}
\item We say that $\Omega\subset \R^N$ is a \textit{domain} if it is open and connected.

\item A bounded domain $\Omega$ is called a \textit{Lipschitz domain}, if for each point $x\in\partial\Omega$ there exists a neighbourhood $U\subset \R^N$ such that the set $U\cap\Omega$ can be presented by the inequality $x_n> f(x_1,\dots,x_{n-1})$ in some Cartesian coordinate system with function $f$ satisfying a Lipschitz condition.

\item We say that a domain $\Omega$ possesses the \textit{outer cone property} if there exist positive numbers $a,b$ such that each point of $\partial\Omega$ is the vertex of a cone contained in $\R^N\setminus\Omega$, where the cone in question is presented by inequalities $x^2_1+\dots+x^2_{n-1}<bx^2_n$, $0<x_n<a$, in some Cartesian coordinate system.
\item A domain $\Omega$ possesses the \textit{inner cone property} if there exist positive numbers $a,b$ such that each point of $\overline{\Omega}$ is the vertex of a cone contained in $\overline{\Omega}$, where the corresponding cone is presented by inequalities $x^2_1+\dots+x^2_{n-1}<bx^2_n$, $0<x_n<a$, in some Cartesian coordinate system.
\item By a \textit{John domain} we call any bounded domain $\Omega$ for which there exist $x_0\in \Omega$ and constant $A>0$ such that each $x\in\Omega$ can be connected to $x_0$ by a rectifiable
curve $\gamma:[0,1]\mapsto\Omega$, $\gamma(0)=x$, $\gamma(1)=x_0$, such that
\begin{equation*}
\ell(\gamma([0,t]))\leq Ad(\gamma(t))
\end{equation*}
for all $t\in[0,1]$, where $\ell(\gamma([0,t]))$ is length of $\gamma([0,t])$.
\item We say that a bounded domain $\Omega$ has the \textit{outer ball portion property}, if there exist positive constants $b$ and $r_0$ such that, for all $x\in\partial\Omega$ and all $r \in \left(0, r_0\right]$, one has
\begin{equation}\label{EQ:mira_kouli}
\frac{\lambda^N\left(B(x,r)\cap (\R^N\setminus\Omega)\right)}{\lambda^N\left(B(x,r)\right)}\geq b.
\end{equation}
(Domains with ball portion property can be also found in \cite{hajlasz1999}, \cite{edmunds2015}, \cite{edmunds2017}.)
\end{itemize}
\end{defn}

Note that if $\Omega$ is a Lipschiz domain, then it possesses the outer cone property and even the inner cone property. Moreover, any bounded domain having inner cone property is also a John domain, and if a bounded domain $\Omega$ possesses the outer cone property, then it has also the outer ball portion property. Converse assertions are not true in general. (See also \cite[Section 1.1.9]{mazja1985} and \cite[Section 3]{edmunds2015}.)

Let $\Omega$ be bounded and enjoying the inner cone property. Then the generalized Poincar\' e inequality (cf.~e.g.~\cite[Section 1.1.11]{mazja1985}) tells us that for each $m$-times weakly differentiable function $u$ such that $\left|\nabla^m u\right| \in L^p(\Omega)$, $p\in[1,\infty)$, one has
$$
\sum_{k=0}^{m-1}\left\||\nabla^k u|\right\|_{L^p(\Omega)}<\infty.
$$
Note that for a general domain this is not true.

Let us now introduce a one more concept of regularity of a domain based on the so-called isoperimetric function, sometimes called also an isoperimetric profile. It will guarantee an important relation similar to the Poincar\' e inequality.

\begin{defn}\label{D:essential-boundary}
Let $E\subset\mathbb R^N$. Then the set $\partial^ME$, defined as the collection of all $x\in\mathbb R^N$ such that
\begin{equation*}
    \lim_{r\to0_+}\frac{1}{\lambda^N(B(x,r))}\int_{B(x,r)}\chi_E(y)\,dy \quad \text{is nether $0$ nor $1$,}
\end{equation*}
is called the \emph{essential part of the boundary of $E$}, see \cite[Definition 5.8.4]{ziemer1989}.
\end{defn}

\begin{defn}[isoperimetric function]
For domain $\Omega\subset\R^N$ we define the \textit{perimeter} of a measurable set $E\subset\Omega$
\begin{equation*}
P(E,\Omega)=
\int_{\Omega\cap \partial^M E}
d\mathcal{H}^{N-1}(x),
\end{equation*}
where $\mathcal{H}^{N-1}$ is the $(N-1)$-dimensional Hausdorff measure.
The \textit{isoperimetric function} (also called \textit{isoperimetric profile}) $I_{\Omega}: [0, \lambda^N(\Omega)] \mapsto [0, \infty]$ of $\Omega$ is then given by
\begin{equation*}
I_{\Omega}(s)=\inf\left\{P(E,\Omega),E\subset\Omega, s\leq\lambda^N(E)\leq\frac{\lambda^N(\Omega)}2\right\} \quad\text{for } s\in\left[0,\frac{\lambda^N(\Omega)}2\right],
\end{equation*}
and $I_{\Omega}(s)=I_{\Omega}(\lambda^N(\Omega)-s)$ for $s\in\left[\frac{\lambda^N(\Omega)}2,\lambda^N(\Omega)\right]$.
\end{defn}

The isoperimetric function of several domains is known. For instance, for any John domain, we have
$$
I_{\Omega}(s)\approx s^{\frac1{N'}}
$$
near $0$.
In Section 4 of \cite{CPS} is given survey of results about this concept. For each domain $\Omega$ it holds that $I_{\Omega}(s)<\infty$ for $s\in\left[0,\frac{\lambda^N(\Omega)}2\right]$, moreover there exists a constant $C = C(\Omega)$ such that
\begin{equation*}
	I_{\Omega}(s)\leq Cs^{\frac1{N'}}
\end{equation*}
for $s$ near $0$. This means that the best possible behavior of an isoperimetric function at
$0$ is that of John domains.
In \cite{CPS} it is moreover proved that for $p\in[1,\infty]$ and for any domain $\Omega$ for which there exists $C>0$ such that
\begin{equation}\label{EQ:isoperimetric_s}
	I_{\Omega}(s)\geq Cs\quad\text{for }s\in\left[0,\frac{\lambda^N(\Omega)}2\right]
\end{equation}
and for every $m$-times weakly differentiable function $u$ such that $\left|\nabla^m u\right| \in L^p(\Omega)$, one has
$$
\sum_{k=0}^{m-1}\left\||\nabla^k u|\right\|_{L^1(\Omega)}<\infty.
$$

To finish this section, we will recall that in the case when $\Omega$ is Lipschitz, an elegant characterization of the space $W_0^{1,p}(\Omega)$ in terms of the trace operator is available.

Let $\Omega\subset\R^N$ be a Lipschitz domain and let $T$ be the classically defined trace operator (see e.g. \cite[Section 6.4]{kufner1977}). Then
$$
W_0^{1,p}(\Omega)=\{u\in W^{1,p}(\Omega),Tu=0\,\,\text{a.e.~in }\partial\Omega\}.
$$

\section{Key theorems}\label{S:proof}
The following theorem will be used in case $p=1$ to prove our main result. We will however present a more general assertion which covers all $p\in[1,\infty)$. Note that, for $p\in (1,\infty)$, this theorem can be obtained as a consequence of \cite[Theorem 3.13]{kinnunen1997}. We present a new and different proof here, based on a~new approach to bounded domains, followed by a classical extension argument that enables one to include unbounded domains. Both proofs, namely the one presented below and that in~\cite{kinnunen1997}, are given for a~general open set in $\R^N$.

\begin{thm}\label{TH:Lpinfty}
Let $p\in [1,\infty)$ and let $\Omega\subset\R^N$ be open set. Let $u$ be a function such that
$$
u\in W^{1,p}(\Omega)\qquad\text{and}\qquad\frac{u}{d}\in L^{p,\infty}_a(\Omega).
$$
Then
$$
u\in W^{1,p}_0(\Omega).
$$
\end{thm}

\begin{proof}
We will assume that neither $\Omega$ nor $\R^N\setminus \Omega$  is empty, as otherwise the assertion is either trivial or well-known, respectively. First, assume that $\Omega$ is bounded.

\textit{Step 1: The distance function $d$ is an element of $W^{1,p}_0(\Omega)$.}
For each $\eta>0$, let us define the function $d_{\eta}$ by $d_{\eta}(x)=(d(x)-\eta)_+$ for $x\in\Omega$. Then $d_{\eta}$ satisfies the Lipschitz condition with constant $1$ and has compact support in $\Omega$. So, $d_\eta\in W^{1,p}_0(\Omega)$.
We have
\begin{equation}\label{EQ:dist1}
\left\|d-d_{\eta}\right\|_{L^p(\Omega)}
\leq
\left\|\eta\right\|_{L^p(\Omega)}
=
\eta\lambda^N(\Omega)^{\frac1p}
\xrightarrow{\eta\rightarrow0+}0,
\end{equation}
and, on employing the dominated convergence theorem, in which the role of the integrable majorant is played by the constant function identically equal to $1$ on $\Omega$, we arrive at
\begin{equation}\label{EQ:dist2}
\left\||\nabla(d-d_{\eta})|\right\|_{L^p(\Omega)}
=
\left\| \chi_{\{x\in\Omega,d(x)\leq\eta\}}|\nabla d|\right\|_{L^p(\Omega)}
\xrightarrow{\eta\rightarrow0+}0.
\end{equation}
Thus, combining \eqref{EQ:dist1} and \eqref{EQ:dist2} we get
$$\lim_{\eta\rightarrow0}\left\|d-d_{\eta}\right\|_{W^{1,p}(\Omega)}=0.$$
Since $W^{1,p}_0(\Omega)$ is a closed subspace of $W^{1,p}(\Omega)$, we obtain $d\in W^{1,p}_0(\Omega)$.

\textit{Step 2: Let $h\in W^{1,p}(\Omega)$. Suppose that there exists $g\in W^{1,p}_0(\Omega)$ such that $0\leq|h(x)|\leq g(x)$ for a.e.~$x\in\Omega$. Then $h\in W^{1,p}_0(\Omega)$.}

To prove this, let us take $\{g_n\}_{n=1}^{\infty}\in C^{\infty}_0(\Omega)$ such that $\left\|g-g_n\right\|_{W^{1,p}(\Omega)}\rightarrow 0$ as $n\rightarrow \infty$. Without loss of generality, we can assume that all the functions $h$ and $g_n$, $n\in\N$, are nonnegative (otherwise we prove the assertion for $h_+$ and $h_-$). Let us define
$$
h_n(x)=\min\{h(x),g_n(x)\}
$$
for each $n\in\N$.
Since each $h_n$ is the minimum of two functions from $W^{1,p}(\Omega)$, we have $h_n\in W^{1,p}(\Omega)$. This follows in a standard manner from~\cite[Corollary~2.1.8]{ziemer1989}.
Moreover, $h_n\leq g_n\in C^{\infty}_0(\Omega)$, so we easily get $h_n\in W^{1,p}_0(\Omega)$. In particular, $h-h_n$ is weakly differentiable.

For $n\in\N$, denote
\begin{align*}
    &E_n=\{x\in\Omega:g(x)>h(x)>g_n(x)\},
        \\
    &F_n=\{x\in\Omega:g(x)=h(x)>g_n(x)\},
\end{align*}
and observe that
\begin{equation}\label{E:alpha}
    \lim_{n\to\infty}\lambda^N(E_n)=0,
\end{equation}
as the converse would contradict $\left\|g-g_n\right\|_{L^{p}(\Omega)}\rightarrow 0$. By the construction of $\{h_n\}_{n=1}^{\infty}$, one has, for every $n\in\N$,
\begin{align}\label{E:beta}
        \|h-h_n\|_{L^p(\Omega)} &= \|h-h_n\|_{L^p(E_n\cup F_n)}
        \le \|h-h_n\|_{L^p(E_n)} + \|h-h_n\|_{L^p(F_n)}
            \nonumber\\
        &\le \|h\|_{L^p(E_n)} + \|g-g_n\|_{L^p(F_n)}
        \le \|h\|_{L^p(E_n)} + \|g-g_n\|_{L^p(\Omega)},
\end{align}
owing to the estimate $0\le h_n\le h$ and the definition of $F_n$. Clearly, since $h\in L^p(\Omega)$, one gets
\begin{equation}\label{E:gamma}
    \lim_{n\to\infty}\|h\|_{L^p(E_n)}=0,
\end{equation}
on employing~\eqref{E:alpha} and using the absolute continuity of the Lebesgue integral. Furthermore, as $g_n\to g$ in $W^{1,p}(\Omega)$, we also have
\begin{equation}\label{E:delta}
    \lim_{n\to\infty}\|g-g_n\|_{L^p(\Omega)}=0.
\end{equation}
Altogether,~\eqref{E:beta}, \eqref{E:gamma} and~\eqref{E:delta} combined yield
\begin{equation}\label{E:epsilon}
    \lim_{n\to\infty}\|h-h_n\|_{L^p(\Omega)}=0.
\end{equation}

Next, one obtains
\begin{align}\label{E:eta}
        \nonumber\||\nabla(h-h_n)|\|_{L^p(\Omega)} &\le \||\nabla(h-h_n)|\|_{L^p(E_n)}+\||\nabla(h-h_n)|\|_{L^p(F_n)}
            \\
        \nonumber&\le \||\nabla(h-g)|\|_{L^p(E_n)}+\||\nabla(g-h_n)|\|_{L^p(E_n)}+\||\nabla(h-h_n)|\|_{L^p(F_n)}
            \\
       \nonumber &\le \||\nabla(h-g)|\|_{L^p(E_n)}+2\||\nabla(g-g_n)|\|_{L^p(E_n\cup F_n)}
            \\
       \nonumber &\le \||\nabla(h-g)|\|_{L^p(E_n)}+2\||\nabla(g-g_n)|\|_{L^p(\Omega)}
                   \\
        &\le \||\nabla(h-g)|\|_{L^p(E_n)}+2\|g-g_n\|_{W^{1,p}(\Omega)}.
\end{align}
Since $|\nabla (h-g)|\in L^p(\Omega)$, one gets, by~\eqref{E:alpha} and using the absolute continuity of the Lebesgue integral, similarly as in~\eqref{E:gamma},
\begin{equation}\label{E:gamma-b}
    \lim_{n\to\infty}\||\nabla(h-g)|\|_{L^p(E_n)}=0.
\end{equation}
Thus, by~\eqref{E:delta},~\eqref{E:eta} and~\eqref{E:gamma-b}, we arrive at
\begin{equation}\label{E:omega}
    \lim_{n\to\infty}\||\nabla(h-h_n)|\|_{L^p(\Omega)}=0.
\end{equation}
Coupling~\eqref{E:epsilon} with~\eqref{E:omega}, we get
\begin{align*}
\lim_{n\rightarrow\infty}\left\|h-h_n\right\|_{W^{1,p}(\Omega)}
=0.
\end{align*}
Once again, due to the fact that $W^{1,p}_0(\Omega)$ is a closed subspace of $W^{1,p}(\Omega)$, we obtain $h\in W^{1,p}_0(\Omega)$.


\textit{Step 3: Construction of an  approximating sequence.}
Let $u$ satisfy the assumptions of the theorem. Without any loss of generality, assume that $u\geq0$. Since $u\in L^{p}(\Omega)$, we have that $u$ is finite for almost every $x\in\Omega$. For each $k\in\N$, let us define
$$
u_k=\min\{u,kd\}\qquad\text{and}\qquad E_k=\{x\in \Omega, u(x)>kd(x)\}.
$$
Then,
\begin{equation}\label{EQ:Ek_zero}
\lim_{k\rightarrow\infty}\lambda^N(E_k)=0
\end{equation}
 and $u_k\leq kd$. We have that
$$
\left\|u_k\right\|_{L^p(\Omega)}\leq \left\|u\right\|_{L^p(\Omega)}
$$
and
$$
\left\||\nabla u_k|\right\|_{L^p(\Omega)}\leq \left\||\nabla u|\right\|_{L^p(\Omega)} + k\lambda^N(\Omega)^{\frac{1}{p}},
$$
and so $u_k\in W^{1,p}(\Omega).$
Thus, applying Step 1 to the function $kd$ and employing Step 2, we get $u_k\in W^{1,p}_0(\Omega)$

\textit{Step 4: The final approach.} Finally, we will show that $\lim_{k\rightarrow\infty}\left\|u-u_k\right\|_{W^{1,p}(\Omega)}=0$, and thereby complete the proof.
We have
\begin{equation}\label{EQ:characteristic}
u-u_k=(u-kd)\chi_{E_k}\leq u \chi_{E_k}.
\end{equation}
Thus, using
\eqref{EQ:Ek_zero} and employing the dominated convergence theorem once again, this time with the majorant $u\in L^p(\Omega)$, we get easily that
\begin{equation}\label{EQ:zero_order}
\lim_{k\rightarrow\infty}\left\|u-u_k\right\|_{L^{p}(\Omega)}=0.
\end{equation}
In order to deal with the gradient of $u-u_k$, we establish the estimate
\begin{equation}\label{EQ:first_order}
\left\||\nabla(u-u_k)|\right\|_{L^{p}(\Omega)}
\leq
\left\|\chi_{E_k}|\nabla u|\right\|_{L^{p}(\Omega)}+\left\|\chi_{E_k}|\nabla (kd)|\right\|_{L^{p}(\Omega)},
\end{equation}
where we used \eqref{EQ:characteristic} and the triangle inequality.
Analogously to the argument which we applied above to obtain \eqref{EQ:zero_order}, using \eqref{EQ:Ek_zero} and the dominated convergence theorem with the majorant $|\nabla u|\in L^p(\Omega)$, we get
\begin{equation}\label{EQ:first_summand}
\lim_{k\rightarrow\infty}\left\|\chi_{E_k}|\nabla u|\right\|_{L^{p}(\Omega)}=0.
\end{equation}
As for the second summand on the right-hand side of~\eqref{EQ:first_order}, we have that $
\left|\nabla (kd)\right|\leq k
$ thanks to the fact that $d$ is a Lipschitz function with constant $1$,
and so
\begin{align}\label{EQ:second_summand}
\lim_{k\rightarrow\infty}\left\|\chi_{E_k}|\nabla (kd)|\right\|_{L^{p}(\Omega)}
&\leq
\lim_{k\rightarrow\infty}k \left\|\chi_{E_k}\right\|_{L^{p}(\Omega)}
=
\lim_{k\rightarrow\infty}k \lambda^N(E_k)^{\frac1p}\nonumber\\
&=
\lim_{k\rightarrow\infty}k \lambda^N\left(\left\{x\in \Omega, \frac{u(x)}{d(x)}>k\right\}\right)^{\frac1p}
=
0
\end{align}
owing to the fact that $\frac{u}{d}\in L^{p,\infty}_a(\Omega)$ and using Proposition~\ref{TH:ACnorm}.
Combining \eqref{EQ:zero_order}, \eqref{EQ:first_order}, \eqref{EQ:first_summand} and \eqref{EQ:second_summand}, we obtain
$$
\lim_{k\rightarrow\infty}\left\|u-u_k\right\|_{W^{1,p}(\Omega)}=0.
$$
Since $W^{1,p}_0(\Omega)$ is closed in $W^{1,p}(\Omega)$, we obtain from Step 3 that $u\in W^{1,p}_0(\Omega)$.

To prove the assertion for unbounded open set $\Omega$, we recall the method used at the end of the proof of \cite[Theorem 3.13]{kinnunen1997}. Let $x_0\in \partial\Omega$. Let us take a sequence of functions $\{\phi_i\}_{i=1}^{\infty}\subset C_0^{\infty}(\R^N)$ such that for each $x\in\R^N$ one has $\phi_i(x)\in[0,1]$,
$\phi_i(x)=1$ for $x\in B(x_0,2^i)$, $\phi_i(x)=0$ for $x\in \R^N\setminus B(x_0,2^{i+1})$, and for each $x\in\R^N$ it holds that $\left|\nabla\phi_i(x)\right|\leq C$, where $C$ is independent of $i$. Let $v_i=\phi_i u$ and denote $\Omega_i=\Omega \cap B(x_0,2^{i+2})$. Then $v_i\in L^{p}(\Omega_i)$,
$$
\left|\nabla v_i\right|\leq \left|\phi_i\right|\left|\nabla u\right|+\left|\nabla\phi_i\right|\left| u\right| \in L^{p}(\Omega_i).
$$
Furthermore, $\frac{v_i(x)}{\operatorname{dist}(x,\partial\Omega_i)}\in L^{p,\infty}_a(\Omega_i)$, since
\begin{equation*}
    \frac{v_i(x)}{\operatorname{dist}(x,\partial\Omega_i)} = 0 \quad \text{for $x\in \Omega_i\setminus B(x_0,2^{i+1})$,}
\end{equation*}
while for $x\in \Omega_i\cap B(x_0,2^{i+1})$ we have
$$
\operatorname{dist}(x,\partial\Omega_i)\leq \operatorname{dist}(x,x_0)< 2^{i+1}\leq \operatorname{dist}(x,\partial\Omega_i\cap \Omega),
$$ 
thus 
$$
\operatorname{dist}(x,\partial\Omega_i)=
\operatorname{dist}(x,\partial\Omega_i\cap (\R^N\setminus \Omega))=
\operatorname{dist}(x,\partial\Omega_i\cap \partial\Omega)\geq \operatorname{dist}(x, \partial\Omega).
$$
Since the opposite inequality 
$$
\operatorname{dist}(x,\partial\Omega_i)\leq \operatorname{dist}(x, \partial\Omega)
$$
is trivial we obtain
$$
\operatorname{dist}(x,\partial\Omega_i)= \operatorname{dist}(x, \partial\Omega).
$$
Then
\begin{equation*}
    \frac{|v_i(x)|}{\operatorname{dist}(x,\partial\Omega_i)} = \frac{|v_i(x)|}{\operatorname{dist}(x,\partial\Omega)} \le \frac{|u(x)|}{d(x)}\in L^{p,\infty}_a(\Omega).
\end{equation*}
Since $\Omega_i$ is bounded, we obtain $v_i\in W^{1,p}_0(\Omega_i)$  and thus also $v_i\in W^{1,p}_0(\Omega)$. Moreover, using the dominated convergence theorem with the majorant $u$, respectively $|\nabla u|$,
$$
\lim_{i\rightarrow\infty}\left\|v_i- u\right\|_{W^{1,p}(\Omega)}
\leq
\lim_{i\rightarrow\infty}\left(\left\|u(\phi_i-1)\right\|_{L^{p}(\Omega)}+\left\||\nabla u|(\phi_i-1)\right\|_{L^{p}(\Omega)}+\left\|u|\nabla\phi_i|\right\|_{L^{p}(\Omega)}\right)
= 0,
$$
hence we conclude from the closedness of the space $W^{1,p}_0(\Omega)$ that $u\in W^{1,p}_0(\Omega)$. The proof is complete.
\end{proof}

If we moreover assume a mild regularity of domain, we obtain, as a consequence, the following theorem.

\begin{thm}\label{TH:with_regularity}
Let $p\in [1,\infty)$ and let $\Omega\subset\R^N$ be a bounded domain satisfying the condition
\begin{equation}\label{E:intersection-condition}
    W^{1,p}(\Omega)\cap W^{1,1}_0(\Omega)=W^{1,p}_0(\Omega).
\end{equation}
Let $u$ be a function such that
$$
u \in L^1(\Omega)\qquad\text{and}\qquad\left|\nabla u\right| \in L^p(\Omega)\qquad\text{and}\qquad\frac{u}{d}\in L^{1,\infty}_a(\Omega).
$$
Then
$$
u\in W^{1,p}_0(\Omega).
$$
\end{thm}

\begin{proof}
First note, that $\left|\nabla u\right| \in L^p(\Omega)$ implies $\left|\nabla u\right| \in L^1(\Omega)$, and thus $u\in W^{1,1}(\Omega)$. Now we employ Theorem~\ref{TH:Lpinfty} for $p=1$, and get thereby
$$
u\in W^{1,1}_0(\Omega).
$$
Since $\Omega$ is bounded, we can take a cube $Q\subset \R^N$ such that $\Omega \subset Q$. Let us denote
$$
u_0(x)=\left\{
\begin{array}{l@{\quad}l}
u(x),& x\in\Omega,\\
0,& x\in Q\setminus\Omega.
\end{array}
\right.
$$
It is not difficult to see now, owing to the characterization of the space $W^{1,1}_0(\Omega)$ in terms of approximation by smooth functions having compact support in $\Omega$, that $u_0\in W^{1,1}_0(Q)$. Moreover $\left|\nabla u_0\right| \in L^p(Q)$. Now, since $Q$ is a Lipschitz domain, we can use the Poincar\' e inequality to obtain that $u_0\in W^{1,p}(Q)$. Consequently,
$$
u\in W^{1,p}(\Omega).
$$
Finally, from the assumptions on $\Omega$ we get
$$
u\in W^{1,p}_0(\Omega).
$$
\end{proof}

Note that the implication $\left|\nabla u\right| \in L^p(\Omega)\Rightarrow u\in W^{1,p}(\Omega)$ does not hold in general, as can be seen in \cite[Section 1.1.4]{mazja1985}. Here, it is enforced by the assumption that the functions in question vanish at the boundary, and thus the extension is easy.

The exact geometrical meaning of the condition~\eqref{E:intersection-condition} is not immediately seen. However, various reasonable sufficient conditions are available in literature (see e.g.~\cite[Section 3]{HK}). More detailed discussion upon this matter will be carried out in Section~\ref{S:domains}.


\section{Relations between domains}\label{S:domains}
Recall that the principal goal of this paper is to prove an equivalence theorem in the spirit of Theorem 5.5 in \cite{edmunds2015} or Theorem 6.1 in \cite{edmunds2017}. For this, we first need to remove the assumption $u\in L^1(\Omega)$ from Theorem~\ref{TH:with_regularity}. This can be done e.g.~by adding some other appropriate assumption on regularity. Then we need to include the second implication, which is given in literature with different assumptions on the regularity of a domain. Therefore, we will now add a section explaining relations of several types of regularity of a domain.

Sufficient conditions for~\eqref{E:intersection-condition}
are given in literature (\cite[Section 3]{HK}). Especially, it is satisfied for Lipschitz domains, bounded domains with outer cone property and much more. We shall now present another one, based on a different point of view.

\begin{thm}\label{TH:domains}
Let $p\in [1,\infty)$ and let $\Omega\subset\R^N$ be a domain satisfying the outer ball portion property~\eqref{EQ:mira_kouli}.
Then~\eqref{E:intersection-condition} holds.
\end{thm}

\begin{proof}
For any bounded set $\Omega$, the inclusion
$$
W^{1,p}(\Omega)\cap W^{1,1}_0(\Omega)\supset W^{1,p}_0(\Omega)
$$
follows from embeddings of Sobolev spaces. Let us prove the converse inclusion. For $p=1$, the assertion is trivial. Let us prove it for $p>1$.

Let $\Omega$ be a domain satisfying the outer ball portion property with corresponding positive constants $b$ and $r_0$ (see Definition~\ref{TH:doamins}). We assume that $u\in W^{1,p}(\Omega)\cap W^{1,1}_0(\Omega)$. Then $\left|\nabla u\right|$ in $L^p(\Omega)$. Thus, using the boundedness of the maximal operator, we get that $M(\left|\nabla u\right|)
\in L^p(\Omega)$. Moreover,
$$
M_{2d(x)}(\left|\nabla u\right|\chi_{B(x,d(x))})(x)
\leq
M_{2d(x)}(\left|\nabla u\right|)(x)
\leq
M(\left|\nabla u\right|)(x)
$$
for each $x\in\Omega,$ 
and, consequently, $M_{2d(x)}(\left|\nabla u\right|\chi_{B(x,d(x))})\in L^1(\Omega)$ owing to $L^p(\Omega)\hookrightarrow L^1(\Omega)$. Now, we apply \cite[Lemma 4.4]{edmunds2015}, which is formulated for domains satisfying the outer ball portion property and for $u\in W^{1,1}_0(\Omega)$. (This lemma originally appeared in \cite{hajlasz1999} for functions from $C_0^{\infty}(\Omega)$.) We get that there exists a positive constant $C(b)$ such that
$$
\frac{\left|u(x)\right|}{d(x)}\leq C(b) M_{2d(x)}(\left|\nabla u\right|\chi_{B(x,d(x))})(x)
$$
for all $x\in\Omega, d(x)<r_0$. Moreover, $\frac{\left|u(x)\right|}{d(x)}\leq \frac{\left|u(x)\right|}{r_0}$ for $x\in\Omega, d(x)\geq r_0$.
Thus,
$$
\frac{u}{d}\in L^1(\Omega).
$$
Together with the fact that $\left|\nabla u\right|$ in $L^p(\Omega)$, on employing \cite[Theorem 5.4]{edmunds2015}, this guarantees that $u\in W^{1,p}_0(\Omega)$, establishing our claim.
\end{proof}

Now let us point out that the outer ball portion property of $\Omega$ is independent of whether or not the condition \eqref{EQ:isoperimetric_s} on isoperimetric profile is satisfied. We will present two examples.

\begin{example}[domain having the outer ball portion property but not $I_{\Omega}(s)\geq Cs$: rooms and passages]\label{EX:balls}
Let $\Omega$ be the union of ``rooms and passages'' such that ball-shaped rooms have radius $2^{-k}$ and passages are of length $2^{-k}$ and width $2^{-4k}$. See Figure~\ref{FIG:koule_a_tunely}. Note that $2^{-4}\pi< \frac{\lambda^2(\Omega)}2<2^{-2}\pi$, in which the upper bound is the volume of the largest room.
\begin{figure}[ht]
  \centering
  \includegraphics[width=7cm]{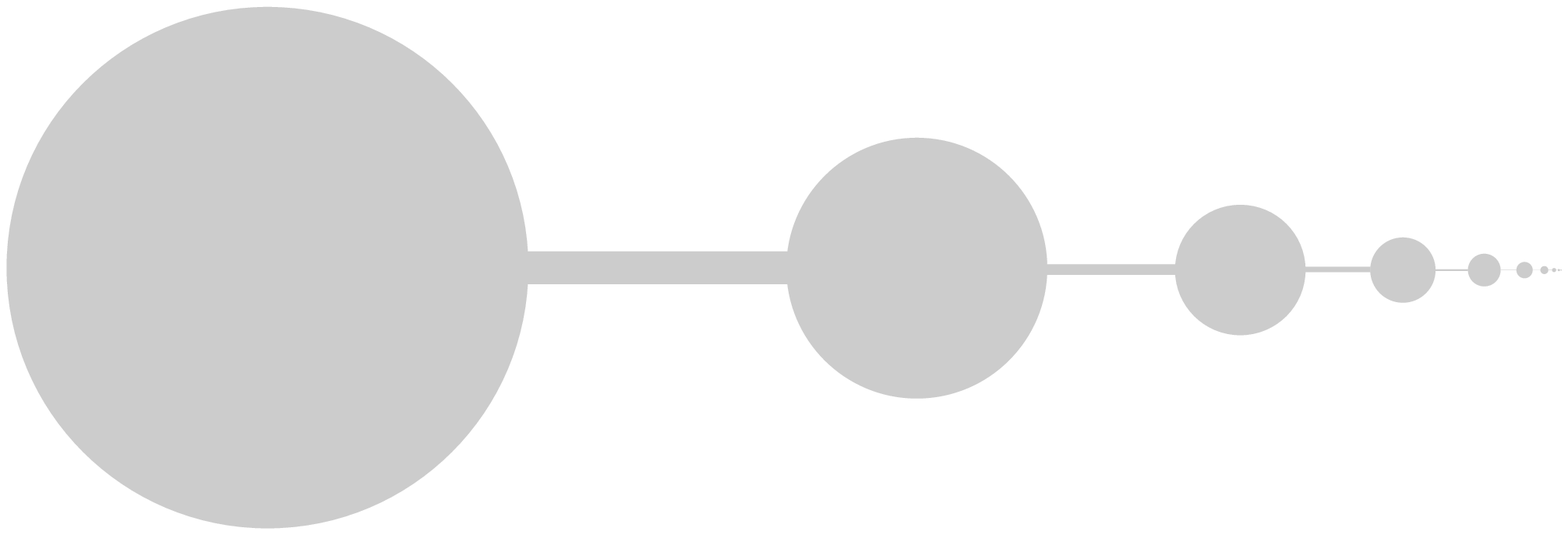}
	\centering
  \caption{Example~\ref{EX:balls}}
  \label{FIG:koule_a_tunely}
\end{figure}

Then it is obvious that $\Omega$ possesses the outer ball portion property, but $I_{\Omega}(s)\leq C s^2$ near $0$ for $C>0$. Indeed, for each $s\in(0,2^{-4}\pi)$, there exists
$k\in\N$ such that
$$
2^{-2(k+1)}\pi\leq s<2^{-2k}\pi,
$$
and we can find $E\subset\Omega$ and $s\leq\lambda^2(E)\leq \frac{\lambda^2(\Omega)}2$ such that $E$ contains a room of measure $2^{-2k}\pi$ and smaller rooms with passages such that
$$
P(E,\Omega)
<
2^{-4(k-1)}
=
2^82^{-4(k+1)}
\leq
\pi^{-2}2^8s^2.
$$
Our claim now follows from the very definition of $I_{\Omega}$.
\end{example}

Note that the core of this example is that we can setup width of passages sufficiently small.

\begin{example}[domain satisfying $I_{\Omega}(s)\geq Cs$ without the outer ball portion property]\label{EX:squares}
Let us set $I_1=\left(\frac12,1\right)$,
$I_k=\left(2^{-k},2^{-k+1}-2^{-2k}\right)$, $k=2,3,\dots$, and
\begin{align*}
M &=\bigcup_{k=1}^{\infty} I_k
\end{align*}
and define the domain $\Omega$ by
$$
\Omega=\left\{(x,y)\in\R^2,-1<x<1,-1<y<\Phi(x)\right\},
$$
in which
\begin{equation*}
    \Phi(x) = \sum_{k=1}^{\infty}\lambda^1(I_k)\chi_{I_k}(x).
\end{equation*}
Then $\Omega$ is a union of squares, see Figure~\ref{FIG:panelaky_ctvercove}.
\begin{figure}[ht]
  \centering
  \includegraphics[width=7cm]{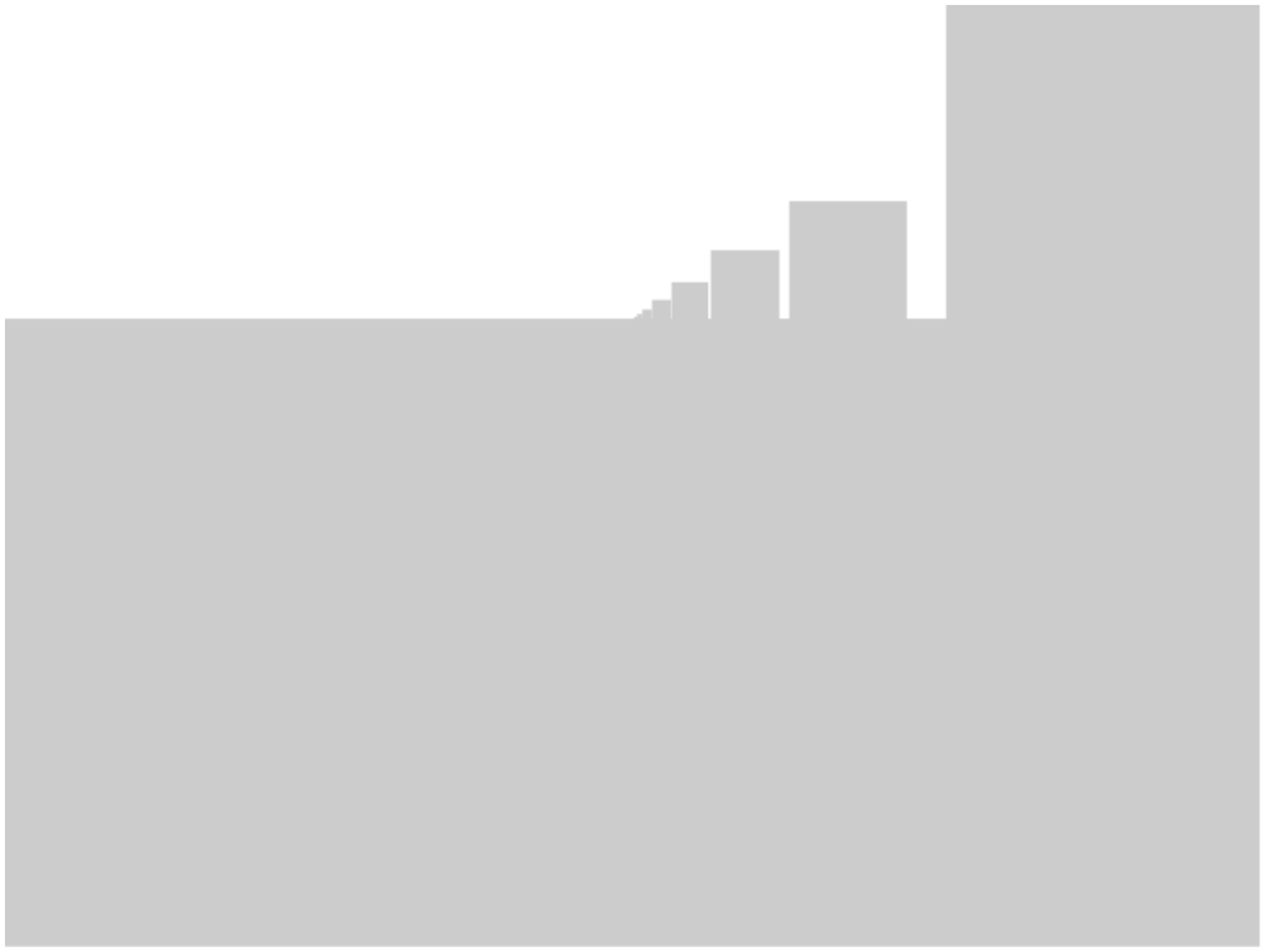}
  \caption{Example~\ref{EX:squares}}
  \label{FIG:panelaky_ctvercove}
\end{figure}

It is easy to see that $\Omega$ satisfies the inner cone property, and thus its isoperimeric profile is $I_{\Omega}(s)\approx \sqrt{s}>s$ for $s>0$ near $0$. Now let us explain why the outer ball portion property is not fulfilled. We will focus on points $z=(x,0)$, $x\in(0,1)\setminus M$. For a fixed $k\in\N$, we find $z_k=(x,0)$ such that $x\in(2^{-k},2^{-(k-1)})\setminus M$, and we take the ball $B(z_k,2^{-k}-2^{-2k})$. Then,
\begin{equation*}
\frac{\lambda^2(B(z_k,2^{-k}-2^{-2k})\cap (\R^2\setminus \Omega))}{\lambda^2(B(z_k,2^{-k}-2^{-2k}))}
\leq
\frac{(2^{-k}-2^{-2k})2^{-2k}}{\pi(2^{-k}-2^{-2k})^2}
=
\frac{1}{\pi(2^{k}-1)}
\xrightarrow{k\rightarrow\infty}0.
\end{equation*}
Now, for each $r_0,b>0$ we can find $k\in \N$ such that $\frac{1}{\pi(2^{k}-1)}<b$ and $2^{-k}-2^{-2k}<r_0$. Finally, following the preceding argument, we find $z\in\partial\Omega$ and $r\in(0,r_0)$ such that
$$
\frac{\lambda^2(B(z,r)\cap (\R^2\setminus \Omega))}{\lambda^2(B(z,r))}<b,
$$
which clearly negates the outer ball portion property.
\end{example}

Now let us introduce for the sake of completeness two more examples of domains which satisfy both the conditions, but at the same time are sufficiently untidy in order not to have the outer cone property. The first of them even has the inner cone property.

\begin{example}\label{EX:crocodile}
Let us define two linear polygonal lines $A=\left\{\left[x,a(x)\right], x\in[0,1]\right\}$ and $B=\left\{\left[x,b(x)\right], x\in[0,1]\right\}$, where functions $a$ and $b$ are given by
\begin{align*}
a(0)=0, a(1)=0, a\left(\frac1{3^k}\right)=0, a\left(\frac2{3^k}\right)= \frac1{3^k}, \text{ linearly connected},\\
b(0)=0, b(1)=-\frac12, a\left(\frac1{3^k}\right)=-\frac12\frac1{3^k}, a\left(\frac2{3^k}\right)= 0, \text{ linearly connected},
\end{align*}
where $k\in\N$.
We define $\Omega$ as the polygon bounded by lines $A$, $\left\{[x,y],x=1,y\in[0,1]\right\}$, $\left\{[x,y],x\in[-1,1],y=1\right\}$, $\left\{[x,y],x=-1,y\in[-1,1]\right\}$, $\left\{[x,y],x\in[-1,1],y=-1\right\}$, $\left\{[x,y],x=1,y\in[-1,-\frac12]\right\}$ and $B$, see Figure~\ref{FIG:crocodile}.
\begin{figure}[ht]
  \centering
  \includegraphics[width=7cm]{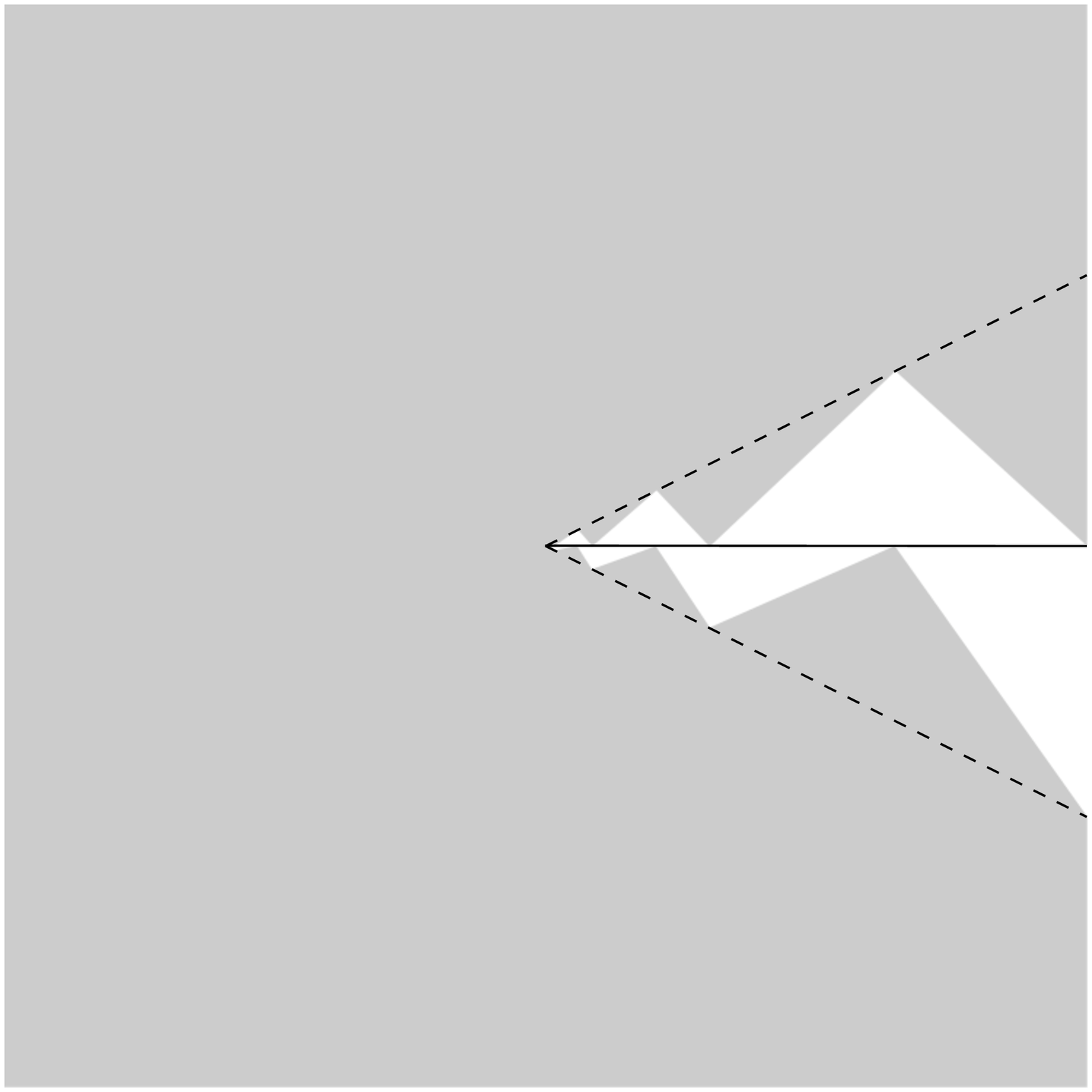}
  \caption{Example~\ref{EX:crocodile}}
  \label{FIG:crocodile}
\end{figure}
Then is easy to see that $\Omega$ satisfies the inner cone property (with the cone of angle $\frac{\pi}4-\varepsilon$ with $\varepsilon\in(0,\frac{\pi}4)$ fixed), and thus its isoperimeric profile is $I_{\Omega}(s)\approx \sqrt{s}>s$ for $s>0$ near $0$. Moreover, it is obvious that $\Omega$ does not satisfy the outer cone property at the point $[0,0]$. Let us comment on why it possesses the outer ball portion property.

It is obvious that the polygon $\Omega_s$, bounded by the lines $\left\{[x,y],x\in[0,1],y=0\right\}$, $\left\{[x,y],x=1,y\in[0,1]\right\}$, $\left\{[x,y],x\in[-1,1],y=1\right\}$, $\left\{[x,y],x=-1,y\in[-1,1]\right\}$, $\left\{[x,y],x\in[-1,1],y=-1\right\}$, $\left\{[x,y],x=1,y\in[-1,-\frac12]\right\}$ and $\left\{[x,y],x\in[0,1],y=-\frac12x\right\}$, possesses the outer cone property, hence also the outer ball portion property. Now, from the formula for the volume of a trapezoid, we see that the corresponding measures needed in this definition are equal for $\Omega_s$ and $\Omega$. For better understanding we can use here squares instead of balls in the definition of the outer ball portion property.
\end{example}

In our last example we shall employ a domain, introduced in \cite{edmunds2015}, which has the outer ball portion property, but does not possess the outer cone property.

\begin{example}\label{EX:skyscrapers}
Let $Q_0=(-1,1)\times(-1,0)$ and for $k\in \N$ we set
$$
Q_k=(2^{-k},2^{-k}+2^{-k-3})\times[0,1).
$$
We define the domain $\Omega$ by
$$
\Omega=\bigcup_{k=0}^{\infty}Q_k,
$$
see Figure~\ref{FIG:panelaky_uzke}.
\begin{figure}[ht]
  \centering
  \includegraphics[width=7cm]{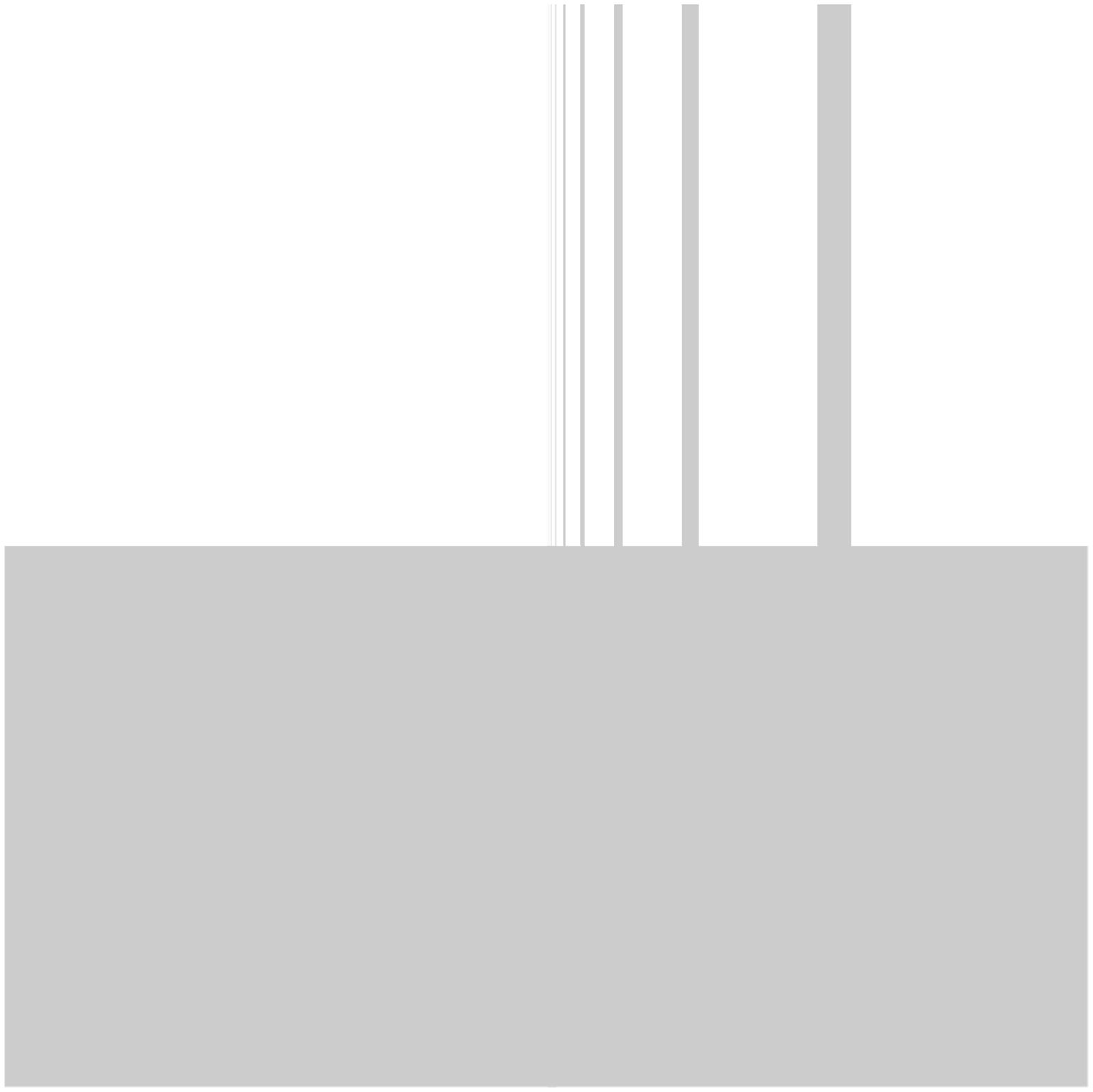}
  \caption{Example~\ref{EX:skyscrapers}}
  \label{FIG:panelaky_uzke}
\end{figure}
Then from \cite[Theorem 3.4]{edmunds2015} we have that $\Omega$ possesses the outer ball portion property and \cite[Theorem 3.5]{edmunds2015} yields that $\Omega$ does not satisfy the outer cone property. Here, we will show that
\begin{equation}\label{E:I-estimate}
    I_{\Omega}(s)\geq \frac{s}{\sqrt{2}}\quad\text{for every $s\in\left(0,\frac{\lambda^2(\Omega)}{2}\right)$.}
\end{equation}
Thus, fix $s\in\left(0,\frac{\lambda^2(\Omega)}{2}\right)$. Take arbitrary $t\in\left[s,\frac{\lambda^2(\Omega)}{2}\right]$ and an arbitrary measurable set $E\subset\Omega$ such that $\lambda^2(E)=t$. Denote further
$$
E_k=E\cap Q_k \quad\text{for $k\in \N\cup\{0\}$.}
$$
It will be useful to note that, since $\lambda^2(E_0)\le\lambda^2(Q_0)=2$, one has
\begin{equation}\label{E:sqrt}
    \sqrt{\lambda^2(E_0)} \ge \frac{\lambda^2(E_0)}{\sqrt 2}.
\end{equation}
Altogether, using subsequently~Lemma~\ref{L:multiperimeter}, the isoperimetric inequality, Lemma~\ref{TH:I_rectangle} with $a=2^{-k-3}$, the estimate \eqref{E:sqrt}, and the fact that $\lambda^2(E_k)\leq\lambda^2(Q_k)=2^{-k-3}$ for every $k\in\N$, we obtain
\begin{align*}
P(E,\Omega)
&\geq
P(E_0,Q_0)+\sum_{k=1}^{\infty}P(E_k,Q_k)\geq
I_{Q_0}(\lambda^2(E_0))+\sum_{k=1}^{\infty}I_{Q_k}(\lambda^2(E_k))\\
&\geq
\sqrt{2\lambda^2(E_0)}+\sum_{k=1}^{\infty}\sqrt{2^{-k-2}}\sqrt{\lambda^2(E_k)}
\geq
\sqrt{2\lambda^2(E_0)}+\sum_{k=1}^{\infty}\sqrt{\frac{\lambda^2(E_k)}{2}}\sqrt{\lambda^2(E_k)}\\
&\geq
\lambda^2(E_0)+\sum_{k=1}^{\infty}\frac{\lambda^2(E_k)}{\sqrt{2}}
\ge
\frac{\lambda^2(E)}{\sqrt{2}}=\frac{t}{\sqrt{2}}
\ge
\frac{s}{\sqrt{2}},
\end{align*}
establishing the desired estimate~\eqref{E:I-estimate}. The proof is complete.
\end{example}

In the corollary that follows, we are going to use the assumption that the domain in question is ``sufficiently regular''. By that we mean the combination of three separate requirements which we shall now explain in detail. First, we require that $\Omega$ obeys~\eqref{E:intersection-condition} (in order to enable us to use Theorem~\ref{TH:with_regularity}). Second, we want that $u\in L^1(\Omega)$ whenever $\left|\nabla u\right|\in L^p(\Omega)$. And, finally, we demand that the ``converse implication'' holds, namely that $\frac{u}{d}\in L^{1,\infty}_a(\Omega)$ provided that $u\in W^{1,p}_0(\Omega)$.

Various sufficient conditions for~\eqref{E:intersection-condition} can be found in~\cite[Section 3]{HK}. Furthermore, due to Theorem~\ref{TH:domains}, we also know that the outer ball portion property~\eqref{EQ:mira_kouli} does the job. The second requirement can be enforced for example by a suitable lower bound for the isoperimetric profile~\eqref{EQ:isoperimetric_s} (see~\cite{CPS}). The third requirement is guaranteed by, for example, the Lipschitz condition (\cite{kadlec1966}), the outer cone condition (\cite{edmunds1987}), the outer ball portion property (\cite{edmunds2015}), or an appropriate capacitary condition (\cite{kinnunen1997}).

We can conclude from the above-given survey of relations between particular conditions that, for example, any domain $\Omega$, obeying simultaneously the outer ball portion property and the estimate  $I_{\Omega}(s)\geq C s$ near $0$, is sufficiently regular, and that neither of these two conditions implies the other one.

\begin{cor}\label{TH:equivalence}
Let $p\in (1,\infty)$ and let $\Omega\subset\R^N$ be a sufficiently regular domain in the sense specified in the preceding paragraph.
Then
$$
\left|\nabla u\right| \in L^p(\Omega)\qquad\text{and}\qquad\frac{u}{d}\in L^{1,\infty}_a(\Omega)
$$
if and only if
$$
u\in W^{1,p}_0(\Omega).
$$
\end{cor}

\begin{proof}
Since $p\in (1,\infty)$ and $\Omega$ is a~bounded domain satisfying~\eqref{E:intersection-condition}
and \eqref{EQ:isoperimetric_s}, the `if' part follows owing to \cite[Corollary 4.3]{CPS}, which ensures $u\in L^1(\Omega)$, and Theorem~\ref{TH:with_regularity}. The `only if' part can be proved using \cite[Remark 3.18]{kinnunen1997} or \cite[Proposition 5.1]{edmunds2015}, which are all stated for $p\in (1,\infty)$, and subsequent applying of embedding theorems between Lorentz spaces combined with Proposition~\ref{TH:L1qsubset}.
\end{proof}

\begin{proof}[Proof of Theorem~\ref{T:main}]
The assertion of Theorem~\ref{T:main} now follows from Corollary~\ref{TH:equivalence} upon a correct interpretation of the sufficient regularity of a domain (one only needs to realize that, as mentioned in the paragraph preceding the corollary, the outer ball portion property accompanied by~\eqref{E:iso} is enough).
\end{proof}

\section{On dimension one}\label{S:dimension1}
The dimension one has a somewhat special position in the research of Sobolev spaces. For the sake of completeness, we will state an application of our result in the dimension one. First, we need an observation, which is natural and similar to other ones used in higher dimensions, but in the dimension one it seems to be slightly neglected. In order to keep the paper self-contained and complete, we include a short elementary proof. The result appears in various forms in literature, but usually with some assumptions slightly overlapping with those we need (see e.g.~\cite{leoni2009} or~\cite[Corolary~V.3.21]{edmunds1987}, where the result is stated for arbitrary order of Sobolev space, but on the other hand only for $p>1$).


\begin{lemma}\label{TH:impl_in_1}
Let $p\in[1,\infty)$, $a,b \in\R$, $a<b$, and $u\in W^{1,p}_0(a,b)$. Then
\begin{equation}\label{EQ:limits}
\lim_{t\rightarrow a+} u(t)=0\qquad\text{and}\qquad\lim_{t\rightarrow b-} u(t)=0.
\end{equation}
\end{lemma}

\begin{proof}
Recall that a function $u\in W^{1,p}(a,b)$, $p\in[1,\infty)$, has a representative from $AC(a,b)$, here denoted $u$ again, such that $u$ is continuous and bounded on $(a,b)$, $u\in BV(a,b)$, and there exist finite limits $\lim_{t\rightarrow a+}u(t)$ and $\lim_{t\rightarrow b-}u(t)$ (survey of these results can be seen in \S2.1, \S3.1 and \S7.2 of Part 1 of the book \cite{leoni2009}).
Moreover, $W^{1,p}(a,b)\hookrightarrow C([a,b])$,~i.e. there exists a constant $C>0$ such that for each $u\in W^{1,p}(a,b)$ we have
$$
\sup_{t\in(a,b)}\left|u(t)\right|\leq C \left\|u\right\|_{W^{1,p}(a,b)}.
$$
Indeed, using the mean value theorem for an absolutely continuous extension of the function $u$ to $[a,b]$, one can find $t_0\in [a,b]$ such that
$$
u(t_0)=\frac1{b-a}\int_{a}^{b}u(t)\,dt.
$$
Then for all $t\in [a,b]$ we have
\begin{align*}
\left|u(t)\right|&=\left|u(t_0)+ \int_{t_0}^{t}u'(\tau)\,d\tau\right|
\leq
(b-a)^{-\frac1{p}}\left(\int_{a}^{b}\left|u(\tau)\right|^p\,d\tau\right)^{\frac1{p}}+(b-a)^{1-\frac1{p}}\left(\int_{a}^{b}\left|u'(\tau)\right|^p\,d\tau\right)^{\frac1{p}}\\
&\leq(b-a)^{-\frac1{p}}\max\{1,b-a\}\left\|u\right\|_{W^{1,p}(a,b)},
\end{align*}
where we used the triangle inequality and the H\"older inequality.

Let $u\in W_0^{1,p}(a,b)$ and denote $\delta =\lim_{t\rightarrow a+}\left|u(t)\right|$. Let us assume $\delta>0$ and take $\varepsilon<\frac{\delta}{C}.$ Since $u\in W_0^{1,p}(a,b)$, we can find $\varphi\in C^{\infty}_0(a,b)$ such that
$$
\left\|u-\varphi\right\|_{W^{1,p}(a,b)}<\varepsilon.
$$
We have
$$
\delta
=
\lim_{t\rightarrow a+}\left|u(t)\right|
=
\lim_{t\rightarrow a+}\left|u(t)-\varphi(t)\right|
\leq
\sup_{t\in(a,b)}\left|u(t)-\varphi(t)\right|
\leq
C \left\|u-\varphi\right\|_{W^{1,p}(a,b)}
\leq
C \varepsilon<\delta,
$$
which is a contradiction. Thus $\lim_{x\rightarrow a+} u(x)=0$. Analogously, $\lim_{x\rightarrow b-} u(x)=0$.
\end{proof}

\begin{proposition}\label{TH:equiv_in_1}
Let $p\in[1,\infty)$,  $a,b \in\R$, $a<b$, and $u\in W^{1,p}(a,b)$. Then
$u\in W^{1,p}_0(a,b)$ if and only if \eqref{EQ:limits} is satisfied.
\end{proposition}

\begin{proof}
The `only if' part is a consequence of Lemma~\ref{TH:impl_in_1}. We will prove the `if' part.

Let $u\in W^{1,p}(a,b)$ be such that \eqref{EQ:limits} is satisfied. Let us denote by $\bar{u}$ the extension of $u$ by $0$ to $\R\setminus (a,b)$.
We can take two functions $\varphi_1,\varphi_2\in C^{\infty}(\R)$ such that $0\leq\varphi_1,\varphi_2\leq1$, $\operatorname{supp}\varphi_1\subset(-\infty,a+2\frac{b-a}{3})$, $\operatorname{supp}\varphi_2\subset(a+\frac{b-a}{3}, \infty)$ and $\varphi_1+\varphi_2=1$ (the partition of unity). Then we denote $u_1=\varphi_1 \bar{u}$ and  $u_2=\varphi_2 \bar{u}$. Obviously $u_1, u_2 \in W^{1,p}(a,b)$, and each satisfies \eqref{EQ:limits}. For every $n\in\N$ let us define
$$
U_n(t)=u_1\left(t-\frac1n\right)\quad\text{for $t\in\R$.}
$$
Then $U_n\in W^{1,p}_0{(a,b)}$ for sufficiently large $n\in \N$, since $U_n \in W^{1,p}(a,b)$ is compactly supported in $(a,b)$, and
$$
\lim_{n\rightarrow\infty}\left\|U_n-u_1\right\|_{W^{1,p}(a,b)}=0
$$
thanks to the $p$-mean continuity of the functions $u_1$ and $u_1'$. Thus, from the closedness of the space $W^{1,p}_0(a,b)$,  we get $u_1\in W^{1,p}_0(a,b)$. Similarly we can prove that $u_2\in W^{1,p}_0(a,b)$. Altogether,
$$
u=u_1+u_2\in  W^{1,p}_0(a,b).
$$
\end{proof}

\begin{thm}\label{TH:intersection}
Let $\Omega\subset\R^N$ be a bounded domain. In dimension $N=1$, the property
$$
W^{1,p}(\Omega)\cap W^{1,1}_0(\Omega)=W^{1,p}_0(\Omega)
$$
is satisfied for each bounded domain $\Omega$. In dimensions $N>1$ this is not true.
\end{thm}

\begin{proof}
In dimension $N=1$ one inclusion follows from Proposition~\ref{TH:equiv_in_1}, and the converse one from the nesting property of Lebesgue spaces on a bounded set.

For higher dimensions a counterexample is given in \cite{HK}, see also the book \cite{Heinonen1993}.
\end{proof}

Note that the part of the assertion of Theorem~\ref{TH:intersection} for $N=1$ is also true on a union of suitable open bounded intervals.

\begin{cor}
Let $p\in (1,\infty)$ and $\Omega\subset\R$ be a bounded domain.
Then
$$
u' \in L^p(\Omega)\qquad\text{and}\qquad\frac{u}{d}\in L^{1,\infty}_a(\Omega)
$$
if and only if
$
u\in W^{1,p}_0(\Omega).
$
\end{cor}

\begin{proof}
One implication is fulfilled thanks to Theorems~\ref{TH:with_regularity} and~\ref{TH:intersection}. Note that the assumption $u\in L^1(a,b)$ of Theorem~\ref{TH:with_regularity} is fulfilled thanks to the absolute continuity of $u$ on $(a,b)$. The converse implication can be proved applying a suitable version of the Hardy inequality, (see e.g. \cite{kadlec1966}, or, for dimension N=1, \cite{turcinova2017}), and using embedding theorems between Lorentz spaces and Proposition~\ref{TH:L1qsubset}.
\end{proof}

\section{Examples}\label{S:examples}
In this section we present several counterexamples presenting that assumptions of the main theorem can not be omitted.
\begin{example}
Conjunction of conditions $\frac{u}{d}\in L^{1,\infty}(\Omega)$ and $\left|\nabla u\right|\in L^p(\Omega)$ is not sufficient for $u\in W^{1,p}_0(\Omega)$ even for regular (e.g. Lipschitz) domains. This can be demonstrated using the following example.

Set $\Omega=(0,1)^N$, $N\in\N$, and $u(x)=1$ for each $x\in \Omega$. The graph of $d(x)$ is a~``pyramid'' with vertex in $(\frac12,\dots,\frac12)$, where $d(x)$ attains the value $\frac{1}{2}$. For example, in one dimension we have
$$
d(x)=\left\{
\begin{array}{l@{\quad}l}
x,& x\in(0,\frac12],\\
1-x,& x\in(\frac12,1).
\end{array}
\right.
$$
Let us compute the distribution function $\xi\mapsto \lambda^N(\{x\in\Omega;\frac{1}{d(x)}>\xi\})$. It is clearly equal to one for $\xi\in[0,2]$, thus it suffices to compute it for $\xi>2$. We obtain
\begin{align*}
\lambda^N(\{x\in\Omega;\frac{1}{d(x)}>\xi\})
&=\lambda^N(\{x\in\Omega;d(x)<1/\xi\})\\
&=\lambda^N(\Omega\setminus\{x\in\Omega;d(x)\geq1/\xi\})\\
&=1-\left(1-\frac2{\xi}\right)^N,
\end{align*}
where $(1-\frac2{\xi})^N$ is the volume of a cube, whose distance from the boundary of $\Omega$ is $\frac1{\xi}$.

Thus, and by the definition of the Lorentz norm via distribution function (see Remark~\ref{TH:Lorentz_eq}), we get
$$
\left\|\frac1d\right\|_{L^{1,\infty}(\Omega)}=\sup_{\xi>0}\xi\lambda^N(\{x\in\Omega;\frac{1}{d(x)}>\xi\})=\max\left\{2,\sup_{\xi>2}\xi\left(1-\left(1-2/\xi\right)^N\right)\right\}.
$$
Changing variables $1-\frac2{\xi}=s$, $s\in(0,1)$, we obtain
$$
\sup_{\xi>2}\xi\left(1-\left(1-2/{\xi}\right)^N\right)=\sup_{s\in(0,1)}\frac{2}{1-s}(1-s^N)=\sup_{s\in(0,1)}2(1+s+\cdots+s^{N-1})=2N,
$$
and, consequently, $\frac{u}{d}$ in $L^{1,\infty}(\Omega)$. Obviously, $\left|\nabla u\right|\in L^p(\Omega).$
However, using standard techniques such as the characterization of $W^{1,p}_0(\Omega)$ by the trace operator, we get $u\notin W^{1,p}_0(\Omega)$.
\end{example}

\begin{example}\label{EX:2}
If a bounded domain $\Omega$ is not sufficiently regular, then the conjunction of conditions $\frac{u}{d}\in L^{1,\infty}_a(\Omega)$ and $\left|\nabla u\right|\in L^p(\Omega)$ is not necessarily sufficient for $u\in W^{1,p}_0(\Omega)$.

Let $N\in\N$, $p>N>1$. Set $\Omega=B(0,1)\setminus\left\{\textbf{o}\right\}\subset \R^N$ and $u(x)=-\left|x\right|+1$. The distance function corresponding to $\Omega$ is given by
$$
d(x)
=
\left\{
\begin{array}{l@{\quad}l}
-\left|x\right|+1,& x\in B(0,1)\setminus B(0,\frac{1}{2}),\\
\left|x\right|, & x\in B(0,\frac{1}{2})\setminus\left\{\textbf{o}\right\}.
\end{array}
\right.
$$
Let us compute the distribution function $\xi\mapsto \lambda^N(\{x\in\Omega;\frac{u(x)}{d(x)}>\xi\})$. It is clearly equal to $\lambda^N(B(0,1))$ for $\xi\in[0,1)$, thus it suffices to compute it for $\xi\geq1$. This case corresponds to $x\in B(0,\frac{1}{2})\setminus\left\{\textbf{o}\right\}$ and $\frac{u(x)}{d(x)}=-1+\frac{1}{\left|x\right|}$. We obtain
\begin{align*}
\lambda^N\left(\left\{x\in\Omega;-1+\frac{1}{\left|x\right|}>\xi\right\}\right)
&=\lambda^N\left(\left\{x\in\Omega;\left|x\right|<\frac{1}{\xi+1}\right\}\right)
=\lambda^N\left(B\left(0,\frac{1}{\xi+1}\right)\right)\\
&=\omega_N\left(\frac1{\xi+1}\right)^N,
\end{align*}
where $\omega_N$ denotes the volume of the unit ball in $\R^N$.
Thus, using once again the definition of the Lorentz norm via distribution function, we get
$$
\left\|\frac{u}{d}\right\|_{L^{1,\infty}(\Omega)}=\sup_{\xi>0}\xi\lambda^N\left(\left\{x\in\Omega;\frac{u(x)}{d(x)}>\xi\right\}\right)
=
\max\left\{\omega_N,\sup_{\xi\geq1}\xi\omega_N\left(\frac1{\xi+1}\right)^N\right\}.
$$
Changing variables $\frac1{\xi+1}=s$, $s\in(0,\frac12]$, we obtain
$$
\sup_{\xi\geq1}\xi\omega_N\left(\frac1{\xi+1}\right)^N
=
\sup_{s\in (0,\frac12]}\frac{1-s}{s}\omega_N s^N
\leq
\omega_N \left(\frac12\right)^{N-1},
$$
and thus $\frac{u}{d}\in L^{1,\infty}(\Omega)$. Moreover, we have that
$$
\lim_{\xi\rightarrow\infty}\xi\lambda^N\left(\left\{x\in\Omega;\frac{u(x)}{d(x)}>\xi\right\}\right)
=
\lim_{\xi\rightarrow\infty}\xi\omega_N\left(\frac1{\xi+1}\right)^N=0
$$
and
\begin{equation*}
    \lim_{\xi\rightarrow0_+}\xi\lambda^N\left(\left\{x\in\Omega;\frac{u(x)}{d(x)}>\xi\right\}\right)
=\lim_{\xi\rightarrow0_+}\xi\omega_N=0,
\end{equation*}
and, consequently, using the definition of the absolutely continuous Lorentz norm via distribution function, $\frac{u}{d}\in L^{1,\infty}_a(\Omega)$. Obviously, $\left|\nabla u\right|\in L^p(\Omega)$, and so the conjunction of conditions of interest is satisfied. Moreover $u\in L^p(\Omega)$ and so $u\in W^{1,p}(\Omega)$. However, a point has positive $p$-capacity in $\R^N$, $p>N$, hence we arrive at a contradiction with the Havin-Bagby theorem (see e.g. \cite{HK}). In conclusion, $u\notin W^{1,p}_0(\Omega)$.
\end{example}

\begin{rem}
The preceding counterexample cannot be used for $p\leq N$ due to the fact that the $p$-capacity of a single point is zero in this case. Then, $W^{1,p}_0(B(0,1))=W^{1,p}_0(B(0,1)\setminus\{\textbf{o}\})$ (see \cite[Theorem 2.43]{Heinonen1993}), and thus any function with limit equal to zero on $\partial B(0,1)$ and $c>0$ in $\textbf{o}$ is an element of $W^{1,p}_0(B(0,1)\setminus\{\textbf{o}\})$.
\end{rem}

\appendix

\section{}

\subsection{Background results about spaces of functions with absolutely continuous norm}\label{S:ACnorm}

In this section we survey several useful results about a structure of spaces $L^{p,\infty}_a(\RR)$, $p\in[1,\infty)$.

We shall need two characterizations of the space $L^{p,\infty}_a(\RR)$, one in terms of the non-increasing rearrangement, and one in terms of the distribution function. We shall present it in Proposition~\ref{TH:ACnorm} below without a proof. A slightly modified part of this result can be found in~\cite[Theorem~8.5.3]{PKJF}, the other part can be easily verified by taking generalized inverses.

\begin{proposition}\label{TH:ACnorm}
Let $p\in[1,\infty)$ and $f\in\mathscr{M}(\RR,\mu)$. Then the following statements are equivalent:

\textup{(a)} $f\in L^{p,\infty}_a(\RR)$,

\textup{(b)} $\lim_{t\rightarrow0_+} t^{\frac{1}{p}}f^*(t)=\lim_{t\rightarrow\infty} t^{\frac{1}{p}}f^*(t)=0$,

\textup{(c)} $\lim_{\xi\rightarrow0_+} \xi\mu(\left\{x\in\mathscr{R}:\left|f(x)\right|>\xi\right\})^{\frac{1}{p}}
			=
			\lim_{\xi\rightarrow\infty} \xi\mu(\left\{x\in\mathscr{R}:\left|f(x)\right|>\xi\right\})^{\frac{1}{p}}=0$.
\end{proposition}

Now we will present several useful relations between function spaces that take part in the main results.

\begin{proposition}\label{TH:L1qsubset}
Let $p,q\in[1,\infty)$. Then $L^{p,q}(\RR)\subset L^{p,\infty}_a(\RR)$.
\end{proposition}

\begin{proof}
Assume $f\in L^{p,q}(\RR)$. Hence, also $f\in L^{p,\infty}(\RR)$. We have
\begin{equation*}
\left(\int_0^{\infty}t^{\frac{q}{p}-1}f^*(t)^q\, d t\right)^{\frac1q}<\infty.
\end{equation*}
Since the Lebesgue integral is absolutely continuous, we obtain
\begin{equation*}
\lim_{\varepsilon\rightarrow 0+}\left(\int_0^{\varepsilon}t^{\frac{q}{p}-1}f^*(t)^q\, d t\right)^{\frac1q}=0
\qquad\text{and}\qquad
\lim_{s\rightarrow \infty}\left(\int_{\frac{s}2}^{\infty}t^{\frac{q}{p}-1}f^*(t)^q\, d t\right)^{\frac1q}=0.
\end{equation*}
Moreover, thanks to the monotonicity of $f^*$, we have
\begin{equation*}
\left(\int_0^{\varepsilon}t^{\frac{q}{p}-1}f^*(t)^q\, d t\right)^{\frac1q}
\geq
 f^*(\varepsilon)\left(\int_0^{\varepsilon}t^{\frac{q}{p}-1}\, d t\right)^{\frac1q}
=
\left(\frac{p}{q}\right)^{\frac1q} \varepsilon^{\frac{1}{p}} f^*(\varepsilon)
\end{equation*}
and
\begin{align*}
\left(\int_{\frac{s}2}^{\infty}t^{\frac{q}{p}-1}f^*(t)^q\, d t\right)^{\frac1q}
&\geq
\left(\int_{\frac{s}2}^{s}t^{\frac{q}{p}-1}f^*(t)^q\, d t\right)^{\frac1q}
\geq
 f^*(s)\left(\int_{\frac{s}2}^{s}t^{\frac{q}{p}-1}\, d t\right)^{\frac1q}\\
&=
\left(\frac{p}{q}\right)^{\frac1q}\left(1-\frac{1}{2^{\frac{q}{p}}}\right)^{\frac1q}s^{\frac{1}{p}} f^*(s).
\end{align*}
Altogether, we obtain
\begin{equation*}
\lim_{\varepsilon\rightarrow 0+}\varepsilon^{\frac{1}{p}} f^*(\varepsilon)=0
\qquad\text{and}\qquad
\lim_{s\rightarrow \infty}s^{\frac{1}{p}} f^*(s)=0,
\end{equation*}
and thus, applying Proposition~\ref{TH:ACnorm}, we get $f\in L^{p,\infty}_a(\RR)$.
\end{proof}

Now let us show that, for $p\in[1,\infty)$, $L^{p,\infty}_a(\RR)$ is strictly bigger then the union of all  $L^{p,q}(\RR)$, $q\in [1,\infty)$.

\begin{proposition}
Let $p\in[1,\infty).$ Then $\bigcup_{q<\infty}L^{p,q}(\RR)\subsetneq L^{p,\infty}_a(\RR)$.
\end{proposition}

\begin{proof}
The inclusion can be obtained directly from Proposition~\ref{TH:L1qsubset}. To refute equality we take a nonnegative function
$$
h(t)=
\left\{
\begin{array}{l@{\quad}l}
\frac{1}{t^{\frac1p} \log \log(\frac{1}{t})},& t\in (0,K),\\
0,& t\in[K,\infty),
\end{array}
\right.
$$
where $K=\min\{\mu(\RR), e^{-e^p}\}$.
Then
\begin{align*}
&h'(t)=\frac{t^{\frac1p-1}\left(1-\frac1p\log(\frac1t)\log\log(\frac1t)\right)}{t^{\frac2p}\log(\frac1t)\log^2\log(\frac1t)}<0
\end{align*}
on $(0, K)$, hence $h$ is decreasing on $(0, K)$ and so non-increasing on $(0,\infty)$. As a consequence of the Sierpi\'nski theorem (can be seen for e.g.~in \cite[Example 18-28]{pfeffer1977}) there exists a function $f\in \M(\RR,\mu)$ such that $f^*=h$. We thus have
\begin{align*}
&\lim_{t\rightarrow 0_+}t^{\frac1p}f^*(t)
=\lim_{t\rightarrow 0_+}t^{\frac1p}h(t)
=\lim_{t\rightarrow 0_+}\frac{1}{\log\log(\frac1t)}
=0.
\end{align*}
Moreover, trivially, $\lim_{t\rightarrow \infty}t^{\frac1p}f^*(t)=0$. Therefore, in view of Proposition~\ref{TH:ACnorm}, $f\in L^{p,\infty}_a(\RR)$.

Fix now $q<\infty$. Then
\begin{align*}
\int_0^{K}t^{\frac{q}{p}-1}f^*(t)^q dt
&=\int_0^{K}t^{\frac{q}{p}-1}\frac{1}{t^{\frac{q}{p}}\log^q\log(\frac1t)} dt
=\int_0^{K}\frac{1}{t\log^q\log(\frac1t)} dt\\
&=\int_{\log(K^{-1})}^{\infty} \frac{1}{\log^q s}ds
=\int_{\log\log(K^{-1})}^{\infty} \frac{e^y}{y^q }dy
=\infty.
\end{align*}
Thus $f\notin L^{p,q}(\RR)$ for each $q\in[1,\infty)$, and so we obtain $f\notin \bigcup_{q<\infty}L^{p,q}(\RR)$.
\end{proof}

\subsection{Results concerning isoperimetric profile}\label{SS:isoperimetric}

In this subsection, our first aim is to establish an inequality between the perimeter of a measurable set and sums of perimeters of its decomposition into countably many disjoint subsets. The estimate is likely to be known and we do not claim it as a new result, but we used it above and we have not been able to find a detailed proof in the existing literature, so we will insert one here for the sake of completeness. We recall that $\partial^{M}E$ denotes the essential part of $\partial E$, introduced in Section~\ref{S:preliminaries}.

\begin{lemma}\label{L:essential-is-boundary}
For every $E\subset \mathbb R^N$, one has $\partial^{M}E\subset \partial E$.
\end{lemma}

\begin{proof}
Fix $x\in\partial^ME$. Then
\begin{equation*}
    \limsup_{r\to0_+}\frac{1}{|B(x,r)|}\int_{B(x,r)}\chi_E(y)\,dy >0.
\end{equation*}
Thus, there is a sequence $\{r_n\}$ such that $r_n\to0$ and $|B(x,r_n)\cap E|>0$. Choose some $x_n\in B(x,r_n)\cap E$. Then $x_n\in E$ for every $n\in\N$, and $x_n\to x$ as $n\to\infty$. Hence, $x\in\overline{E}$.

Assume that $x\in E^{\circ}$. Then one can find $r_0>0$ such that $B(x,r_0)\subset E$. But then we have
\begin{equation*}
    \lim_{r\to0_+}\frac{1}{|B(x,r)|}\int_{B(x,r)}\chi_E(y)\,dy=1.
\end{equation*}
This contradicts $x\in \partial^ME$, whence $x\notin E^{\circ}$. Consequently, $x\in \overline{E}\setminus E^{\circ} = \partial E$, as desired.
\end{proof}

\begin{lemma}\label{L:subdomains}
Let $\Omega$, $\Omega_1$ be open and such that $\Omega_1\subset \Omega$. Let $E$ be a measurable subset of $\Omega$. Denote $E_1=\Omega_1\cap E$. Then,
\begin{equation*}
    \Omega_1\cap\partial^ME_1 \subset \Omega\cap \partial^{M}E.
\end{equation*}
\end{lemma}

\begin{proof}
Let $x\in\Omega_1\cap\partial^ME_1$. Since $x\in \Omega_1$ and $\Omega_1$ is open, there is an $\varepsilon>0$ such that $B(x,\varepsilon)\subset\Omega_1$. Moreover, $x\in\partial^ME_1$ implies $x\in\partial E_1$ owing to Lemma~\ref{L:essential-is-boundary}, hence there is a sequence $\{x_n\}$ of elements of $E_1\subset E$ such that $x_n\to x$. We get $x\in\overline{E}$.

Assume, for a time being, that $x\in E^{\circ}$. Take $r_0>0$ such that $B(x,r_0)\subset E$. Clearly, for every $r\in(0,\min\{\varepsilon,r_0\})$, one has $B(x,r)\subset E\cap \Omega_1=E_1$. Consequently, $x\in(E_1)^{\circ}$, which however contradicts $x\in\partial E_1$. This implies that $x\notin E^{\circ}$, and so $x\in\partial E$. Further, since $x\in\Omega_1\subset\Omega$, one also has $x\in\Omega\cap\partial E$.

Now, we prove that $x\in\partial^ME$. Assume that
\begin{equation*}
    \lim_{r\to0_+}\frac{|E\cap B(x,r)|}{|B(x,r)|} = 0.
\end{equation*}
Since $E_1\subset E$, we have
\begin{equation*}
    \lim_{r\to0_+}\frac{|E_1\cap B(x,r)|}{|B(x,r)|} = 0.
\end{equation*}
Thus, $x\notin \partial^ME_1$. Finally, assume that
\begin{equation*}
    \lim_{r\to0_+}\frac{|E\cap B(x,r)|}{|B(x,r)|} = 1.
\end{equation*}
Since, for every $r\in(0,\varepsilon)$, one has $E\cap B(x,r)\subset E\cap\Omega_1=E_1$, we get $E\cap B(x,r)=E_1\cap B(x,r)$, which enforces
\begin{equation*}
    \lim_{r\to0_+}\frac{|E_1\cap B(x,r)|}{|B(x,r)|} = 1.
\end{equation*}
Thus, once again, $x\notin\partial^ME_1$. In each case we arrived at a contradiction. This yields that
\begin{equation*}
    \lim_{r\to0_+}\frac{|E\cap B(x,r)|}{|B(x,r)|}\quad \text{is nether $0$ nor $1$,}
\end{equation*}
in other words, $x\in \partial ^ME$. We conclude that $x\in\Omega\cap\partial^ME$.
\end{proof}

\begin{lemma}\label{L:multiperimeter}
Let $\Omega_i\subset\R^N$, $i\in\N$, be open, pairwise disjoint, and such that $\Omega=\left(\bigcup_{i=1}^{\infty}\overline{\Omega_i}\right)^{\circ}$. Assume that $E\subset \Omega$ is measurable. Denote $E_i=E\cap\Omega_i$ for $i\in\N$. Then,
\begin{equation*}
    P(E,\Omega) \ge \sum_{i=1}^{\infty} P(E_i,\Omega_i).
\end{equation*}
\end{lemma}

\begin{proof}
Since $\Omega_i\cap \partial ^ME_i\subset\Omega\cap\partial E$, and the sets $\Omega_i\cap\partial^ME_i$ are pairwise disjoint, we obtain
\begin{align*}
    P(E,\Omega) &= \int_{\Omega\cap\partial^ME}\,d\mathcal H^{N-1}(x) \ge \int_{\bigcup(\Omega_i\cap\partial^ME_i)}\,d\mathcal H^{N-1}(x)
        \\
    &= \sum_{i=1}^{\infty}\int_{\Omega_i\cap\partial^ME_i}\,d\mathcal H^{N-1}(x) = \sum_{i=1}^{\infty}P(E_i,\Omega_i).
\end{align*}
\end{proof}

The second goal of this subsection is to state and prove a lemma concerning isoperimetric profile of a rectangle, quite natural and likely to be of independent interest, and which moreover proved to be useful in the analysis of Example~\ref{EX:skyscrapers}.

\begin{lemma}\label{TH:I_rectangle}
Let $a\in(0,1)$ and let $Q$ be a planar rectangle congruent to $(0,1)\times(0,a)$. Then
$$
I_Q(s)\geq \sqrt{2as}\text{ for }s\in\left[0,\frac{a}{2}\right].
$$
\end{lemma}

\begin{proof}
    With no loss of generality we may assume that $Q=(0,1)\times(0,a)$. Fix $t\in[0,\frac{a}{2}]$. Then a classical argument can be used to obtain that
    \begin{equation}\label{E:lower-bound-isoperimetric}
        P(E,Q) \ge P(E_t,Q),
    \end{equation}
    for every measurable subset $E$ of $Q$ such that $\lambda^2(E)=t$, in which
    \begin{equation*}
        E_t=
        \begin{cases}
            \left\{[x,y]\in\R^2: x^2+y^2< \frac{4t}{\pi},\ 0<x,y<\sqrt{\frac{4t}{\pi}}\right\} &\text{if $t\in[0,\frac{a^2}{\pi}]$}
                \\
            \left\{[x,y]\in\R^2: x\in\left(0,\frac{t}{2}\right), y\in(0,a)\right\} &\text{if $t\in(\frac{a^2}{\pi},\frac{a}{2}]$.}
        \end{cases}
    \end{equation*}
    Then
    \begin{equation*}
        P(E_t,Q) =
        \begin{cases}
            \sqrt{\pi t} &\text{if $t\in[0,\frac{a^2}{\pi}]$,}
            \\
            a &\text{if $t\in\big(\frac{a^2}{\pi},\frac{a}{2}\big]$,}
        \end{cases}
    \end{equation*}
    whence the function $\Psi$, defined by
    \begin{equation*}
        \Psi(t) = P(E_t,Q) \quad\text{for $t\in\left[0,\frac{a}{2}\right]$,}
    \end{equation*}
    is nondecreasing on $[0,\frac{a}{2}]$. By~\eqref{E:lower-bound-isoperimetric} and the monotonicity of $\Psi$,
    \begin{equation*}
        I_Q(s)
            = \inf\left\{P(E,Q):\lambda^2(E)\in\left[s,\frac{a}{2}\right]\right\} \ge \inf\left\{\Psi(t):t\in\left[s,\frac{a}{2}\right]\right\} = \Psi(s)\quad\text{for $s\in\left[0,\frac{a}{2}\right]$.}
    \end{equation*}
    Thus, for $s\in[0,\frac{a^2}{\pi}]$, we get, owing to the fact that $a<1$,
    \begin{equation*}
        I_Q(s) = \sqrt{\pi s} \ge \sqrt{\pi as} > \sqrt{2as},
    \end{equation*}
    while, for $s\in\big(\frac{a^2}{\pi},\frac{a}{2}\big]$, we have
    \begin{equation*}
        I_Q(s) = a \ge \sqrt{2as}.
    \end{equation*}
    The desired inequality now follows from the latter two estimates.
\end{proof}

\end{document}